\documentclass[12pt]{amsart}

\usepackage{amsmath,amssymb,amsthm}
\usepackage{enumerate}
\usepackage[english]{babel}

\usepackage{pstricks}
\usepackage{graphicx}
\usepackage{curve2e}

\makeatletter
\@namedef{subjclassname@2010}{%
  \textup{2010} Mathematics Subject Classification}
\makeatother

\theoremstyle{plain}
\newtheorem{main-theorem}{Theorem}
\newtheorem{theo}{Theorem}[section]
\newtheorem{prop}[theo]{Proposition}
\newtheorem{lemm}[theo]{Lemma}

\theoremstyle{definition}
\newtheorem{defi}[theo]{Definition}
\newtheorem{exam}[theo]{Example}
\newtheorem{rema}[theo]{Remark}

\newtheorem{ques}[theo]{Problem}
\newtheorem*{clai-nn}{Claim}

\renewcommand{\leq}{\leqslant}
\renewcommand{\geq}{\geqslant}

\newcommand{\bbR}{\mathbb{R}}

\newcommand{\bbC}{\mathbb{C}}
\newcommand{\bbD}{\mathbb{D}}
\newcommand{\Kc}{\mathcal{K}}

\newcommand{\Dc}{\mathcal{D}}

\begin{document}

\baselineskip=17pt

\title[Fibers and local connectedness of planar continua]{Fibers and local connectedness of planar continua}

\author[B. Loridant]{Beno\^it Loridant}
\address{Montanuniversit\"at Leoben,
    Franz Josefstrasse 18\\ Leoben 8700, Austria}
\email{benoit.loridant@unileoben.ac.at}
\author[J. Luo]{Jun Luo}
\address{School of Mathematics\\
    Sun Yat-Sen University\\ Guangzhou 512075, China}
\email{luojun3@mail.sysu.edu.cn}

\date{}

\begin{abstract}
We describe non-locally connected planar continua via the concepts of fiber and numerical scale.

Given a continuum $X\subset\bbC$ and $x\in\partial X$, we show that the set of points $y\in \partial X$ that cannot be separated from $x$ by any finite set $C\subset \partial X$ is a continuum. This continuum is called the {\em modified fiber} $F_x^*$ of $X$ at $x$. If $x\in X^o$, we set $F^*_x=\{x\}$. For $x\in X$, we show that $F_x^*=\{x\}$ implies that $X$ is locally connected at $x$. We also give a concrete planar continuum $X$, which is locally connected at a point $x\in X$ while the fiber $F_x^*$ is not trivial.

The scale $\ell^*(X)$ of non-local connectedness is then the least integer $p$ (or $\infty$ if such an integer does not exist) such that for each $x\in X$ there exist $k\le p+1$ subcontinua $$X=N_0\supset N_1\supset N_2\supset\cdots\supset N_{k}=\{x\}$$ such that $N_{i}$ is a fiber of $N_{i-1}$ for $1\le i\le k$. If  $X\subset\bbC$ is an unshielded continuum or a continuum whose complement has finitely many components, we obtain that local connectedness of $X$ is equivalent to the statement $\ell^*(X)=0$.

We discuss the relation of our concepts to the works of Schleicher (1999) and Kiwi (2004). We further define an equivalence relation $\sim$ based on the fibers and show that the quotient space $X/\sim$ is a locally connected continuum. For connected Julia sets of polynomials and more generally for unshielded continua, we obtain that every prime end impression is contained in a fiber. Finally, we apply our results to examples from the literature and construct for each $n\ge1$ concrete examples of path connected continua $X_n$ with $\ell^*(X_n)=n$.
\end{abstract}

\subjclass[2010]{54D05, 54H20, 37F45, 37E99.}
\keywords{ Local connectedness, fibers, numerical scale, upper semi-continuous decomposition.}

\maketitle

\section{Introduction and main results}

Motivated by the construction of Yoccoz puzzles used in the study on local connectedness of quadratic Julia sets and the Mandelbrot set $\mathcal{M}$, Schleicher~\cite{Schleicher99a} introduces the notion of fiber for full continua (continua  $M\subset\bbC$ having a connected complement $\bbC\setminus M$), based on ``separation lines'' chosen from particular countable dense sets of external rays that land on  points of $M$. Kiwi \cite{Kiwi04} uses finite ``cutting sets'' to define a modified version of fiber for Julia sets, even when they are not connected.

Jolivet-Loridant-Luo~\cite{JolivetLoridantLuo0000} replace Schleicher's ``separation lines'' with ``good cuts'', {\em i.e.}, simple closed curves $J$ such that $J\cap \partial M$ is finite and $J\setminus M\ne\emptyset$. In this way, Schleicher's approach is generalized to continua $M\subset\bbC$ whose complement $\bbC\setminus M$ has finitely many components. For such a continuum $M$, the {\em pseudo-fiber $E_x$} (of $M$) at a point $x\in M$ is the collection of the points $y\in M$ that cannot be separated from $x$ by a good cut; the {\em fiber $F_x$ at $x$} is the component of $E_x$ containing $x$. Here, a point $y$ is separated from a point $x$ by a simple closed curve $J$ provided that $x$ and $y$ belong to different components of $\bbC\setminus J$. And $x$ may belong to the bounded or unbounded component of $\bbC\setminus J$.

Clearly, the fiber $F_x$ at $x$ always contains $x$. We say that a pseudo-fiber or a fiber is {\em trivial} if it coincides with the single point set $\{x\}$.

By \cite[Proposition 3.6]{JolivetLoridantLuo0000}, every fiber of $M$ is again a continuum with finitely many complementary components. Thus the hierarchy by ``fibers of fibers'' is well defined. Therefore, the \emph{scale} $\ell(M)$ \emph{of non-local connectedness} is defined as the least integer $k$ such that for each $x\in M$ there exist $p\le k+1$ subcontinua $M=N_0\supset N_1\supset\cdots\supset N_{p}=\{x\}$ such that $N_{i}$ is a fiber of $N_{i-1}$ for $1\le i\le p$. If such an integer $k$ does not exist we set $\ell(M)=\infty$.

In this paper, we rather follow Kiwi's approach~\cite{Kiwi04} and define ``modified fibers'' for continua on the plane. The key point is: Kiwi focuses on Julia sets and uses ``finite cutting sets'' that consist of pre-periodic points, but we consider arbitrary continua $M$ on the plane (which may have interior points) and use ``finite separating sets''.
We refer to Example~\ref{cutting-2} for the difference between separating and cutting sets. Moreover, in Jolivet-Loridant-Luo~\cite{JolivetLoridantLuo0000}, a good cut is not contained entirely in the underlying continuum $M$. In the curent paper we will remove this assumption and only require that a good cut is a simple closed curve intersecting $\partial M$ at finitely many points.  After this slight modification we can establish the equivalence between the above mentioned two approaches to define fiber, using good cuts or using finite separating sets. See  Remark \ref{modified}  for further details.

The notions and results will be presented in a way that focuses on the general topological aspects, rather than in the framework of complex analysis and dynamics.

\begin{defi}\label{kiwi-fiber}  Let $X\subset\bbC$ be a continuum. We will say that a point $x\in \partial X$ is \emph{separated from a point} $y\in \partial X$ \emph{by a subset} $C\subset X$  if there is a \emph{separation} $\partial X\setminus C=A\cup B$ with $x\in A$ and $y\in B$. Here `` $\partial X\setminus C=A\cup B$ is a separation'' means that $\overline{A}\cap B=A\cap\overline{B}=\emptyset$.
\begin{itemize}
\item The \emph{ modified pseudo-fiber $E^*_x$ of $X$}  at a point $x$ in the interior $X^o$ of $X$ is $\{x\}$; and the \emph{modified pseudo-fiber $E^*_x$ of $X$}  at a point $x\in \partial X$ is the set of the points $y\in \partial X$ that cannot be separated from $x$ by any finite set $C\subset \partial X$.
\item The \emph{modified fiber} $F_x^*$ of $X$
at $x$ is the connected component of $E^*_x$ containing $x$. We say $E^*_x$ or $F_x^*$ is \emph{trivial} if it consists of the point $x$ only. ({\em We will show that $E^*_x=F^*_x$ in Theorem \ref{main-1}, so  the notion of modified pseudo-fiber is only used as a formal definition.})
\item We inductively define a {\em fiber of order} $k\ge2$ as a fiber of a continuum $Y\subset X$, where $Y$ is a fiber of order $k-1$.
\item The {\em local scale of non-local connectedness} of $X$ at a point $x\in X$, denoted $\ell^*(X,x)$, is the least integer $p$ such that there exist $k\le p+1$ subcontinua $$X=N_0\supset N_1\supset N_2\supset\cdots\supset N_{k}=\{x\}$$ such that $N_{i}$ is a fiber of $N_{i-1}$ for $1\le i\le k$. If such an integer does not exist we set $\ell^*(X,x)=\infty$.
\item The (global) {\em scale of non-local connectedness} of $X$ is 
$$\ell^*(X)=\sup\{\ell^*(X,x): x\in X\}.$$ We also call $\ell^*(X,x)$ the {\em local NLC-scale of $X$ at $x$}, and $\ell^*(X)$ the {\em global NLC-scale}.
\end{itemize}
\end{defi}

We firstly obtain the equality $F_x^*=E_x^*$ and relate trivial fibers to local connectedness. Here, local connectedness at a particular point does not imply trivial fiber. In particular, let $\Kc\subset[0,1]$ be Cantor's ternary set, let $X$ be the union of $\Kc\times[0,1]$ with $[0,1]\times\{1\}$. See Figure \ref{comb}.  Then $X$ is locally connected at every $x=(t,1)$ with $t\in\Kc$, while the modified fiber $F_x^*$ at this point is the whole segment $\{t\}\times[0,1]$. See Example \ref{cantor-comb} for more details.
\begin{main-theorem}\label{main-1}
Let $X\subset\bbC$ be a continuum. Then $F_x^*=E_x^*$ for every $x\in X$; moreover, $F_x^*=\{x\}$ implies that $X$ is locally connected at $x$.
\end{main-theorem}

Secondly, we characterize modified fibers $F^*_x=E_x^*$ through simple closed curves $\gamma$ that separate $x$ from points $y$ in $X\setminus F^*_x$ and that intersect $\partial X$ at a finite set or an empty set.

This provides an equivalent way to develop the theory of fibers, for planar continua, and leads to a partial converse for the second part of Theorem \ref{main-1}. See Remark \ref{modified}.

\begin{main-theorem}\label{criterion}
Let  $X\subset\bbC$ be a continuum. Then $F^*_x$ at any point $x\in X$ consists of the points $y\in X$ such that every simple closed curve $\gamma$ separating $x$ from $y$ must intersect $\partial X$ at infinitely many points. Or, equivalently, $X\setminus F_x^*$ consists of the points $z\in X$ which may be separated from $x$ by a simple closed curve $\gamma$ such that $\gamma\cap\partial X$ is a finite set.
\end{main-theorem}
This criterion can be related to Kiwi's characterization of fibers~\cite[Corollary 2.18]{Kiwi04}, as will be explained at the end of Section~\ref{proof-criterion}.
\begin{rema}\label{modified}
We define a simple closed curve $\gamma$ to be a \emph{ good cut of a continuum} $X\subset\bbC$ if $\gamma\cap\partial X$ is a finite set (the empty set is also allowed). We also say that  two points $x,y \in X$ are separated by a good cut $\gamma$  if they lie in different components of $\bbC\setminus \gamma$. This slightly weakens the requirements on ``good cuts'' in \cite{JolivetLoridantLuo0000}. Therefore, given a continuum $X\subset\bbC$ whose complement has finitely many components, the modified pseudo-fiber $E_x^*$ at any point $x\in X$ is a subset of the pseudo-fiber $E_x$ at $x$, if $E_x$ is defined as in \cite{JolivetLoridantLuo0000}. Consequently, we can infer that local connectedness of $X$ implies triviality of all the fibers $F_x^*$, by citing two of the four equivalent statements of \cite[Theorem 2.2]{JolivetLoridantLuo0000}: (1) $X$ is locally connected; (2) every pseudo-fiber $E_x$ is trivial.  The same result does not hold when the complement $\bbC\setminus X$ has infinitely many components. Sierpinski's universal curve gives a counterexample.
\end{rema}

\begin{rema}\label{two-approaches}
The two approaches, via pseudo-fibers $E_x$ and modified pseudo-fibers $E_x^*$, have their own merits. The former one follows Schleicher's approach and is more closely related to the theory of puzzles in the study of Julia sets and the Mandelbrot set; hence it may be used to analyse the structure of such continua by cultivating the dynamics of polynomials. The latter approach has a potential to be extended to the study of general compact metric spaces; and, at the same time, it is directly connected with the first approach when restricted to planar continua.
\end{rema}

Thirdly, we study the topology of $X$ by constructing an equivalence relation $\sim$ on $X$ and  cultivating the quotient space $X/\!\sim$, which will be shown to be a locally connected continuum. This relation $\sim$ is based on fibers of $X$ and every fiber $F_x^*$ is contained in a single equivalence class.

\begin{defi}\label{def:eqrel} Let  $X\subset\bbC$ be a continuum.
 Let $X_0$ be the union of all the nontrivial fibers $F_x^*$ for $x\in X$ and $\overline{X_0}$ denote the closure  of $X_0$. We define $x\sim y$ if $x=y$ or if $x\ne y$  belong to the same component of $\overline{X_0}$. Then $\sim$ is a closed equivalence relation on $X$ such that, for all $x\in X$, the equivalence class $[x]_\sim$ always contains the modified fiber $F_x^*$ and equals  $\{x\}$ if only $x\in (X\setminus \overline{X_0})$. Consequently, every equivalence class $[x]_\sim$ is a continuum, so that the natural projection $\pi(x)=[x]_\sim$ is a monotone mapping, from $X$ onto its quotient $X/\!\sim$.
\end{defi}

\begin{rema} Actually, there is a more natural equivalence relation $\approx$ by defining $x \approx y$ whenever there exist points $x_1=x,$ $x_2,\ldots, x_n=y$ in $X$ such that $x_i\in F^*_{x_{i-1}}$. However, the relation $\approx$ may not be closed, as a subset of the product $X\times X$. On the other hand, if we take the closure of $\approx$ we will obtain a closed relation, which is reflexive and symmetric but may not be transitive (see Example~\ref{relations}). The above Definition~\ref{def:eqrel} solves this problem.
\end{rema}

The following theorem provides important information about the topology of $X/\!\sim$.

\begin{main-theorem}\label{main-3}
 Let  $X\subset\bbC$ be a continuum. Then $X/\!\sim$ is metrizable and is a locally connected continuum, possibly a single point.
\end{main-theorem}

\begin{rema}
The result of Theorem \ref{main-3} is of fundamental significance from the viewpoint of topology. It also plays a crucial role in the study of complex dynamics. In particular, if $J$ is the Julia set (assumed to be connected) of a polynomial $f(z)$ with degree $n\ge2$ the restriction $\left.f\right|_J: J\rightarrow J$ induces a continuous map $f_\sim: J/\!\sim\rightarrow J/\!\sim$ such that $\pi\circ f=f_\sim\circ\pi$. See Theorem \ref{dynamic}. Moreover, the modified fibers $F_x^*$ are closely related to impressions of prime ends. See Theorem \ref{impression}. Combining this with laminations on the unit circle $S^1\subset\bbC$, the system $f_\sim: J/\!\sim\rightarrow J/\!\sim$ is also a factor of the map $z\mapsto z^n$ on $S^1$. However, it is not known yet whether the the decomposition $\{[x]_\sim: x\in X\}$ by classes of $\sim$ coincide with the finest locally connected model discussed in \cite{BCO11}. For more detailed discussions related to the dynamics of polynomials, see for instance \cite{BCO11,Kiwi04} and references therein.
\end{rema}

Finally, to conclude the introduction, we propose two problems.

\begin{ques}\label{scale-model}
To estimate the scale $\ell^*(X)$ from above for particular continua $X\subset\bbC$ such that $\bbC\setminus X$ has finitely many components, and to compute the quotient space $X/\!\sim$ or the locally connected model introduced in \cite{BCO11}. The Mandelbrot set or the Julia set of an infinitely renormalizable quadratic polynomial (when this Julia set is not locally connected) provide very typical choices of $X$. In particular, the scale $\ell^*(X)$ will be zero if the Mandelbrot set is locally connected, \emph{i.e.}, if \emph{MLC} holds. In such a case, the relation $\sim$ is trivial and its quotient is immediate.
\end{ques}

\begin{rema}
Section \ref{examples} gives several examples of continua $X\subset\bbC$. We obtain the decomposition $\{[x]_\sim:x\in X\}$ into sub-continua  and represent the quotient space $X/\!\sim$ on the plane. For those examples, the scale $\ell^*(X)$ is easy to determine.
\end{rema}

\begin{ques}\label{finest} Given an unshielded continuum $X$ in the plane,
is it possible to construct the ``finest'' upper semi-continuous decomposition of $X$ into sub-continua that consist of fibers, in order that the resulted quotient space is  a locally connected continuum (or has the two properties mentioned in Theorem~\ref{main-3})? Such a finest decomposition has no refinement that has the above properties. If $X$ is the Sierpinski curve, which is not unshielded, the decomposition $\left\{[x]_\sim:x\in X\right\}$  obtained in Theorem \ref{main-3} does not suffice.
\end{ques}

\begin{rema}
The main motivation for Problem \ref{finest} comes from \cite{BCO11} in which the authors, Blokh-Curry-Oversteegen, consider ``unshielded'' continua $X\subset\bbC$ which coincide with the boundary of the unbounded component of $\bbC\setminus X$. They obtain the existence of the finest monotone map $\varphi$ from $X$ onto a locally connected continuum on the plane, such that $\varphi(X)$ is the finest locally connected model of $X$ and extend $\varphi$ to a map $\hat{\varphi}: \hat{\bbC}\rightarrow\hat{\bbC}$ that maps $\infty$ to $\infty$, collapses only those components of $\bbC\setminus X$ whose boundary is collapsed by $\varphi$, and is a homeomorphism elsewhere in $\hat{\bbC}\setminus X$ \cite[Theorem 1]{BCO11}. This is of significance in the study of complex polynomials with connected Julia set, see \cite[Theorem 2]{BCO11}.
\end{rema}

\begin{rema}
The equivalence classes $[x]_\sim$ obtained in this paper give a concrete upper semi-continuous decomposition of an arbitrary continuum $X$ on the plane, with the property that the quotient space $X/\!\sim$ is a locally connected continuum. In the special case $X$ is unshielded, the finest decomposition in~\cite[Theorem 1]{BCO11} is finer than or equal to our decomposition $\{[x]_\sim:x\in X\}$. See Theorem~\ref{impression} for details when $X$ is assumed to be unshielded. The above Problem~\ref{finest} asks whether those two decompositions actually coincide. If the answer is yes, the quotient space $X/\!\sim$ in Theorem \ref{main-3} is exactly the finest locally connected model of $X$, which shall be in some sense ``computable''. Here, an application of some interest is to study the locally connected model of an infinitely renormalizable Julia set~\cite{Jiang00} or of the Mandelbrot set, as mentioned in Problem \ref{scale-model}.
\end{rema}

We arrange our paper as follows. Section~\ref{basics} recalls some basic notions and results from topology that are closely related to local connectedness.  Sections~\ref{proof-1},~\ref{proof-criterion} and~\ref{proof-3} respectively prove Theorems~\ref{main-1},~\ref{criterion} and~\ref{main-3}. Section~\ref{fiber-basics} discusses basic properties of fibers, studies fibers from a viewpoint of dynamic topology (as proposed by Whyburn \cite[pp.130-144]{Whyburn79}) and relates the theory of fibers to the theory of prime ends for unshielded continua. Finally, in Section~\ref{examples}, we illustrate our results through examples from the literature and give an explicit sequence of path connected continua $X_n$ satisfying $\ell^*(X_n)=n$.

\section{A Revisit to Local Connectedness}\label{basics}

\begin{defi}\label{lc-notion}
A topological space $X$ is \emph{locally connected at} a point $x_0\in X$ if for any neighborhood $U$ of $x_0$ there exists a connected neighborhood $V$ of $x_0$ such that $V\subset U$, or equivalently, if the component of $U$ containing $x_0$ is also a neighborhood of $x_0$. The space $X$ is then called \emph{locally connected} if it is locally connected at every of its points.
\end{defi}
We focus on metric spaces and their subspaces. The following characterization can be found as the definition of locally connectedness in~\cite[Part A, Section XIV]{Whyburn79}.

\begin{lemm}\label{local-criterion}
A metric space $(X,d)$ is locally connected at $x_0\in X$ if and only if for any $\varepsilon>0$ there exists $\delta>0$ such that any point $y\in X$ with $d(x_0,y)<\delta$ is contained together with $x_0$ in a connected subset of $X$ of diameter less than $\varepsilon$.
\end{lemm}

When $X$ is compact, Lemma \ref{local-criterion} is a local version of \cite[p.183, Lemma 17.13(d)]{Milnor06}. For the convenience of the readers, we give here the concrete statement as a lemma.

\begin{lemm}\label{global-criterion}
A compact metric space $X$ is locally connected if and only if for every $\varepsilon>0$ there exists $\delta>0$ so that any two points of distance less than $\delta$ are contained in a connected subset of $X$ of diameter less than $\varepsilon$.
\end{lemm}

Using Lemma \ref{local-criterion}, we obtain a fact concerning continua of the Euclidean space $\mathbb{R}^n$.

\begin{lemm}\label{fact-1}
Let $X\subset\mathbb{R}^n$ be a continuum and $U=\bigcup_{\alpha\in I}W_\alpha$ the union of any collection $\{W_\alpha: \alpha\in I\}$ of components of $\mathbb{R}^n\setminus X$. If $X$ is locally connected at $x_0\in X$, then so is $X\cup U$. Consequently, if $X$ is locally connected, then  so is $X\cup U$.
\end{lemm}

\begin{proof}
Choose $\delta$ with properties from Lemma \ref{local-criterion} with respect to $x_0, X$ and $\varepsilon/2$. For any $y\in U$ with $d(x_0,y)<\delta$ we consider the segment $[x_0,y]$ between $x_0$ and $y$. If $[x_0,y]\subset(X\cup U)$, we are done. If not, choose the point $z\in ([x_0,y]\cap X)$ that is closest to $y$. Clearly, the segment $[y,z]$ is contained in $X\cup U$. By the choice of $\delta$ and Lemma \ref{local-criterion}, we may connect $z$ and $x_0$ with a continuum $A\subset X$ of diameter less than $\varepsilon/2$. Therefore, the continuum $B:=A\cup[y,z]\subset(X\cup U)$ is of diameter at most $\varepsilon$ as desired.
\end{proof}

In the present paper, we are mostly interested in continua on the plane, especially continua $X$ which are on the boundary of a continuum $M\subset\bbC$. Typical choice of such a continuum $M$ is the filled Julia set of a rational function. Several fundamental results from Whyburn's book~\cite{Whyburn79} will be very helpful in our study.

The first result gives a fundamental fact about a continuum failing to be locally connected at one of its points. The proof can be found in~\cite[p.124, Corollary]{Whyburn79}.

\begin{lemm}\label{non-lc}
A continuum $M$ which is not locally connected at a point $p$ necessarily fails to be locally connected at all points of a nondegenerate subcontinuum of $M$.
\end{lemm}

The second result will be referred to as {\bf Torhorst Theorem} in this paper (see~\cite[p.124, Torhorst Theorem]{Whyburn79}
and~\cite[p.126, Lemma 2]{Whyburn79}).

\begin{lemm}\label{torhorst}

The boundary $B$ of each component $C$ of the complement of a locally connected continuum $M$ is itself a locally connected continuum. If further $M$ has no cut point, then $B$  is a simple closed curve.

\end{lemm}

We finally recall a {\em Plane Separation Theorem} \cite[p.120, Exercise 2]{Whyburn79}.
\begin{prop}\label{theo:sep} If $A$ is a continuum and $B$ is a closed connected set of the plane with $A\cap B=T$ being a totally disconnected set, and with $A\setminus T$ and $B\setminus T$ being connected, then there exists a simple closed curve $J$ separating $A\setminus T$ and $B\setminus T$ such that $J\cap (A\cup B)\subset A\cap B=T$.
\end{prop}

\section{Fundamental properties of fibers}\label{proof-1}

The proof for Theorem \ref{main-1} has two parts. We start from the equality $E_x^*=F_x^*$.

\begin{theo}\label{E=F}
Let $X\subset\bbC$ be a continuum. Then $E_x^*=F_x^*$ for every $x\in X$.
\end{theo}

\begin{proof}

Suppose that $E_x^*\setminus F_x^*$ contains some point $x'$. Then we can fix a separation $E_x^*=A\cup B$ with $F_x^*\subset A$ and $x'\in B$. Since $E_x^*$ is a compact set, the distance ${\rm dist}(A,B):=\min\{|y-z|: y\in A, z\in B\}$ is positive. Let
\[ A^*=\left\{z\in\bbC: {\rm dist}(z, A)<\frac{1}{3}{\rm dist}(A,B)\right\}
\]
and
\[ B^*=\left\{z\in\bbC: {\rm dist}(z, B)<\frac{1}{3}{\rm dist}(A,B)\right\}.\]
Then $A^*$ and $B^*$ are disjoint open sets in the plane, hence $K=X\setminus(A^*\cup B^*)$ is a compact subset of $X$. As $E_x^*\cap K=\emptyset$, we may find for each $z\in K$ a finite set $C_z$ and a separation
\[X\setminus C_z=U_z\cup V_z\]
such that $x\in U_z, z\in V_z$. Here, we have $U_z=X\setminus(C_z\cup V_z)=X\setminus(C_z\cup \overline{V_z})$ and $V_z=X\setminus(C_z\cup U_z)=X\setminus(C_z\cup \overline{U_z})$; so both of them are open in $X$.
By flexibility of $z\in K$, we obtain an open cover $\{V_z: z\in K\}$ of $K$, which then has a finite subcover $\left\{V_{z_1},\ldots, V_{z_n}\right\}$.
Let
\[U=U_{z_1}\cap\cdots\cap U_{z_n},\ V=V_{z_1}\cup\cdots\cup V_{z_n}.\]
Then $U,V$ are disjoint sets open in $X$ such that $C:=X\setminus(U\cup V)$ is a subset of $C_{z_1}\cup\cdots\cup C_{z_n}$, hence it is also a finite set. Now, on the one hand, we have a separation $X\setminus C= U\cup V$ with $x\in U$ and $K\subset V$; on the other hand, from the equality $K=X\setminus(A^*\cup B^*)$ we can infer $U\subset (A^*\cup B^*)$. Combining this with the fact that $x'\in B^*$, we may check that $A':=(U\cup\{x'\})\cap A^*=U\cap A^*$ and $B':=(U\cup\{x'\})\cap B^*=(U\cap B^*)\cup\{x'\}\subset B^*$ are separated in $X$. Let $C'=C\setminus\{x'\}$. Since $A'\subset U$ and $V$ are also separated in $X$, we see that \[X\setminus C'=U\cup \{x'\}\cup V=(U\cap A^*)\cup(U\cap B^*)\cup\{x'\}\cup V=A'\cup(B'\cup V)\]
is a separation with $x\in A'$ and $x'\in(B'\cup V)$.  This contradicts the assumption that $x'\in E_x^*$, because $E_x^*$ being the pseudo-fiber at $x$, none of its points can be separated from $x$ by the finite set $C'$.
\end{proof}

Then we recover in fuller generality  that triviality of the fiber at a point $x$ in a continuum $M\subset \bbC$ implies local connectedness of $M$ at $x$. More restricted versions of this result appear earlier: in \cite{Schleicher99a} for  continua in the plane with connected complement, in \cite{Kiwi04} for Julia sets of monic polynomials or the components of such a set, and in \cite{JolivetLoridantLuo0000} for continua in the plane whose complement has finitely many components.

\begin{theo}\label{trivial-fiber}
If $F_x^*=\{x\}$ for a point $x$ in a continuum $X\subset\bbC$ then $X$ is locally connected at $x$.
\end{theo}
\begin{proof}
We will prove that if $X$ is not locally connected at $x$ then $F_x^*$ contains a non-degenerate continuum $M\subset X$.

By definition, if $X$ is not locally connected at $x$ there exists a number $r>0$ such that the component $Q_x$ of $B(x,r)\cap X$ containing $x$ is not a neighborhood of $x$ in $X$. Here
\[B(x,r)=\{y: |x-y|\le r\}.\]
This means that there exist a sequence of points $\{x_k\}_{k=1}^\infty\subset X\setminus Q_x$ such that $\lim\limits_{k\rightarrow\infty}x_k=x$. Let $Q_k$ be the component of $B(x,r)\cap X$ containing $x_k$. Then $Q_i\cap\{x_k\}_{k=1}^\infty$ is a finite set for each $i\ge1$, and hence we may assume, by taking a subsequence, that $Q_i\cap Q_j=\emptyset$ for $i\ne j$.

Since the hyperspace of the nonempty compact subsets of $X$ is a compact metric space under Hausdorff metric, we may further assume that there exists a continuum $M$ such that $\lim\limits_{k\rightarrow\infty}Q_k=M$ under Hausdorff distance. Clearly, we have $x\in M\subset Q_x$. The following Lemma \ref{reach-boundary} implies that the diameter of $M$ is at least $r$. Since every point $y\in M\setminus\{x\}$ cannot be separated from $x$ by a finite set in $X$, $F_x^*$ cannot be trivial and our proof is readily completed.
\end{proof}

\begin{lemm}\label{reach-boundary}
In the proof for Theorem \ref{trivial-fiber}, every component of $B(x,r)\cap X$ intersects $\partial B(x,r)$. In particular, $Q_k\cap\partial B(x,r)\ne\emptyset$ for all $k\ge1$.
\end{lemm}
\begin{proof}
Otherwise, there would exist a component $Q$ of  $B(x,r)\cap X$ such that $Q\cap\partial B(x,r)=\emptyset$. Then, for each point $y$ on $X\cap\partial B(x,r)$, the component $Q_y$ of $X\cap B(x,r)$ containing $y$ is disjoint from $Q$. By definition of quasi-components, we may choose a separation $X\cap B(x,r)=U_y\cup V_y$ with $Q_y\subset U_y$ and $Q\subset V_y$. Since every $U_y$ is open in $X\cap B(x,r)$, we have an open cover
\[\left\{U_y: y\in X\cap\partial B(x,r)\right\}\]
for $X\cap\partial B(x,r)$, which necessarily has a finite subcover, say $\left\{U_{y_1},\ldots, U_{y_t}\right\}$. Let
\[U=U_{y_1}\cup\cdots\cup U_{y_t},\quad V=V_{y_1}\cap\cdots\cap V_{y_t}.\]
Then $X\cap B(x,r)=U\cup V$ is a separation with $\partial B(x,r)\subset U$. Therefore,
\[X=[(X\setminus B(x,r))\cup U]\cup V\]
is a separation, which contradicts the connectedness of $X$.
\end{proof}

\section{Schleicher's and Kiwi's approaches unified}\label{proof-criterion}

Let $X$ be a topological space and $x_0$ a point in $X$. The component of $X$ containing $x_0$ is the maximal connected set $P\subset X$ with $x_0\in P$. The quasi-component of $X$ containing $x_0$ is defined to be the set
\[Q=\{y\in X: \ \text{no\ separation}\ X=A\cup B\ \text{exists\ such\ that}\ x\in A, y\in B\}.\]
Equivalently, the quasi-component of a point $p\in X$ may be defined as the intersection of all closed-open subsets of $X$ containing $p$. Since any component is contained in a quasi-component, and since quasi-components coincide with the components whenever $X$ is compact \cite{Kuratowski68}, we can infer an equivalent definition of pseudo fiber as follows.

\begin{prop}\label{seppoints} Let $X\subset\bbC$ be a continuum. Two points of $X$ are separated by a finite set $C\subset X$ iff they belong to distinct quasi-components of $X\setminus C$.
\end{prop}

The following proposition implies Theorem~\ref{criterion}. We present it in this form, since it can be seen as a modification of Whyburn's plane separation theorem (Proposition~\ref{theo:sep}). Actually, the main idea of our proof is borrowed from~\cite[p.126]{Whyburn79} and is slightly adjusted.

\begin{prop}\label{key} Let $C$ be a finite subset of a continuum $X\subset\bbC$ and $x,y$ two  points on $X\setminus C$. If there is a separation  $X\setminus C=P\cup Q$ with $x\in P$ and $y\in Q$ then  $x$ is separated from $y$ by a simple closed curve $\gamma$ with $(\gamma\cap X)\subset C$.
\end{prop}

\begin{proof}
We first note that every component of $\overline{P}$ intersects $C$. If on the contrary a component $W$ of $\overline{P}\subset(P\cup C)$ is disjoint from $C$, then $\overline{P}$ is disconnected and a contradiction follows. Indeed, we have $\overline{P}\ne W$ since $\overline{P}\cap C\ne \emptyset$. And all the components of $\overline{P}$ intersecting $C$ are disjoint from $W$. As $C$ is a finite set, there are finitely many such components, say $W_1, \ldots, W_t$. However, since a quasi-component of a compact metric space is just a component, we can find separations $\overline{P}=A_i\cup B_i$ for $1\le i\le t$ such that $W\subset A_i, W_i\subset B_i$. Let $A=\cap_iA_i$ and $B=\cup_i B_i$. Then $\overline{P}=A\cup B$ is a separation with $A\cap Q=\emptyset$, hence $X= A\cup(B\cup Q)$ is a separation of $X$. This contradicts the connectedness of $X$.

Since every component of $\overline{P}$ intersects $C$ and since $C$ is a finite set, we know that $\overline{P}$ has finitely many components, say $P_1,\ldots, P_k$. We may assume that $x\in P_1$. Similarly, every component of $\overline{Q}\subset(Q\cup C)$ intersects $C$ and $\overline{Q}$ has finitely many components, say $Q_1,\ldots, Q_l$. We may assume that $y\in Q_1$.

Let $P_1^*=P_2\cup\cdots\cup P_k\cup Q_1\cup\cdots\cup Q_l$. Then $X=P_1\cup P_1^*$, $x\in P_1$, $y\in P_1^*$ and $(P_1\cap P_1^*)\subset C$.
Let $N_1=\{z\in P_1;\ {\rm dist}(z,P_1^*)\ge1\}$ and for each $j\ge2$, let \footnote{This idea is inspired from the proof of Whyburn's plane separation theorem, see Proposition~\ref{theo:sep}}
\[N_j=\{z\in P_1:\ 3^{-j}\le {\rm dist}(z,P_1^*)\le 3^{-j+1}\}.\]
Clearly, every $N_j$ is a compact set. Therefore, we may cover $N_j$ by finitely many open disks centered at a point in $N_j$ and with radius $r_j=3^{-j-1}$, say $B(x_{j1},r_j),\ldots, B(x_{jk(j)},r_j)$.

For $j\geq 1$, let us set $M_j=\bigcup_{i=1}^{k(j)}\overline{B(x_{ji},r_j)}$. Then
$M=\overline{\bigcup_{j\geq 1}M_j}$ is a compact set containing $P_1$. Its interior $M^o$ contains $x$. Moreover, $P_1^*\cap \left(\bigcup_{j\geq 1}M_j\right)=\emptyset$ by definition of $N_j$ and $M_j$, while $M\setminus\left(\bigcup_{j}M_j\right)$ is a subset of $P_1\cap P_1^*$,  hence we have $M\cap P_1^*=P_1\cap P_1^*$ and  $y\notin M$. Also, $\partial M\cap X$ is a subset of $P_1\cap P_1^*$, hence it is a finite set.

Now $M$ is a continuum, since $P_1$ is itself a continuum and the disks $B(x_{ji},r_j)$ are centered at $x_{ji}\in N_j$. The continuum $M$ is even locally connected at every point on $M\setminus C=\bigcup_{j}M_j$. Indeed, it is locally a finite union of disks, since $M_j\cap M_k=\emptyset$ as soon as $|j-k|>1$ and since every point of $M\setminus C$ is in one of these disks. As $C$ is finite, it follows from Lemma~\ref{non-lc} that $M$ is a locally connected continuum.

Now, let $U$ be the component of $\bbC\setminus M$ that contains $y$.  By Torhorst Theorem, see Lemma \ref{torhorst}, the boundary $\partial U$ of $U$ is a locally connected continuum. Therefore, by Lemma~\ref{fact-1}, the union $U\cup\partial U$  is also a locally connected continuum. Since $U$ is a complementary component of $\partial U$, the union $U\cup\partial U$ even has no cut point. It follows from Torhorst Theorem that the boundary $\partial V$ of any component $V$ of $\bbC\setminus(U\cup\partial U)$ is a simple closed curve. Note that this curve separates every point of $U$ from any point of $V$. Choosing $V$ to be the component of $\bbC\setminus(U\cup\partial U)$ containing $x$, we obtain a simple closed curve $J=\partial V$ separating $y$ from $x$.

Finally, since $J=\partial V\subset\partial U\subset\partial M$, we see that $J\cap X$ is contained in the finite set $C$. Consequently, $J$ is a good cut of $X$ separating $x$ from $y$.
\end{proof}

This result proves Theorem~\ref{criterion} and is related to Kiwi's characterization of fibers. Restricting to connected Julia sets $J(f)$ of polynomials $f$, Kiwi~\cite{Kiwi04} had defined for $\zeta\in J(f)$ the fiber $\textrm{Fiber}(\zeta)$ as the set of $\xi\in J(f)$ such that $\xi$ and $\zeta$ lie in the same connected component of $J(f)\setminus Z$ for every finite set $Z\subset J(f)$, made of periodic or preperiodic points that are not in the grand orbit of a Cremer point. Kiwi showed in~\cite[Corollary 2.18]{Kiwi04} that these fibers can be characterized by using separating curves involving external rays.

\section{A locally connected model for the continuum $X$}\label{proof-3}

In this section, we recall a few notions and results from Kelley's {\em General Topology}~\cite{Kelley55} and construct a proof for Theorem~\ref{main-3}, the results of which are divided into two parts:
\begin{itemize}
\item[(1)] $X/\!\sim$ is metrizable, hence is a compact connected metric space, {\em i.e.}, a continuum.
\item[(2)] $X/\!\sim$ is a locally connected continuum.
\end{itemize}

A decomposition $\Dc$ of a topological space $X$ is {\em upper semi-continuous} if for each $D\in\Dc$ and each open set $U$ containing $D$ there is an open set $V$ such that $D\subset V\subset U$ and $V$ is the union of members of $\Dc$ \cite[p.99]{Kelley55}. Given a decomposition $\Dc$, we may define a projection $\pi: X\rightarrow\Dc$ by setting $\pi(x)$ to be the unique member of $\Dc$ that contains $x$. Then, the quotient space $\Dc$ is equipped with the largest topology such that $\pi: X\rightarrow \Dc$ is continuous. We copy the result of \cite[p.148, Theorem 20]{Kelley55} as follows.

\begin{theo}\label{kelley-p148-20}
Let $X$ be a topological space, let $\Dc$ be an upper semi-continuous decomposition of $X$ whose members are compact, and let $\Dc$ have the quotient topology. Then $\Dc$ is, respectively, Hausdorff, regular, locally compact, or has a countable base, provided $X$ has the corresponding property.
\end{theo}

Urysohn's metrization theorem \cite[p.125, Theorem 16]{Kelley55} states that a regular $T_1$-space whose topology has a countable base is metrizable. Combining this with Theorem \ref{kelley-p148-20}, we see that the first part of Theorem \ref{main-3} is implied by the following theorem, since $X$ is a continuum on the plane and has all the properties mentioned in Theorem \ref{kelley-p148-20}. One may also refer to \cite[p.40, Theorem 3.9]{Nadler92}, which states that any upper semi-continuous decomposition of a compact metric space is metrizable.

\begin{theo}\label{usc}
The decomposition $\{[x]_\sim:x\in X\}$ is upper semi-continuous.
\end{theo}
\begin{proof}
Given a set $U$ open in $X$, we need to show that the union $U_\sim\subset U$ of all the classes $[x]_\sim\subset U$ is open in $X$. In other words, we need to show that $X\setminus U_\sim$ is closed in $X$, which implies that $\pi(X\setminus U_\sim)=(X/\!\sim)\setminus\pi(U_\sim)$ is closed in the quotient $X/\!\sim$. Here, we note that $X\setminus U_\sim$ is just the union of all the classes $[x]_\sim$ that intersects $X\setminus U$.

Assume that $y_k\in X\setminus U_\sim$ is a sequence converging to $y$, we will show that $[y]_\sim\setminus U\ne\emptyset$, hence that $y\in X\setminus U_\sim$. Let $z_k$ be a point in $[y_k]_\sim\setminus U$ for each $k\ge1$. By coming to an appropriate subsequence, we may further assume that
\begin{itemize}
\item $[y_i]_\sim\cap[y_j]_\sim=\emptyset$ for $i\ne j$;
\item the sequence of continua $[y_k]_\sim$ converges to a continuum $M$ under Hausdorff metric;
\item the sequence $z_k$ converges to a point $z_\infty$.
\end{itemize}
Clearly, we have $z_\infty\in M$; and, as $X\setminus U$ is compact, we also have $z_\infty\in X\setminus U$. If the sequence $[y_k]_\sim$ is finite, then $M=[y]_\sim$, thus $[y]_\sim\setminus U\ne\emptyset$. If the sequence $[y_k]_\sim$ is infinite, let us check that $M\subset F_y^*$. Indeed, for any point $z\in M$ and for any finite set $C\subset X$ disjoint from $\{y,z\}$, all but finitely many $[y_k]_\sim$ are connected disjoint subsets of $X\setminus C$. It follows that there exists  no separation $X\setminus C=A\cup B$ such that $y\in A, z\in B$, because $y$ and $z$ are both limit points of the sequence of continua $[y_k]_\sim$. Hence $M\subset F_y^*\subset [y]_\sim$ and $z_\infty\in M\cap(X\setminus U)$, indicating that $[y]_\sim\setminus U\ne\emptyset$.
\end{proof}

\begin{theo}\label{lc}
The quotient $X/\!\sim$ is a locally connected continuum.
\end{theo}
\begin{proof}
As $X$ is a continuum, $\pi(X)=X/\!\sim$ is itself a continuum. We now prove that this quotient is locally connected. If $V$ is an open set in  $X/\!\sim$ that contains $[x]_{\sim}$, as an element of $X/\!\sim$, then the pre-image $U:=\pi^{-1}(V)$ is open in $X$ and contains the class $[x]_{\sim}$ as a subset. We shall prove that $V$ contains a connected neighborhood of $[x]_\sim$. Without loss of generality, we assume that $U\ne X$. Let $Q$ be the component of $U$ that contains $[x]_{\sim}$. By the boundary bumping theorem~\cite[Theorem 5.7, p75]{Nadler92} (see also~\cite[p.41, Exercise 2]{Whyburn79}), since $X$ is connected, we have $\overline{Q}\setminus U\ne\emptyset$. Moreover, our proof will be completed by the following claim.

\bigskip

\noindent
{\bf Claim.} The connected set $\pi(Q)$, hence the component of $V$ that contains $[x]_{\sim}$ as a point, is a neighborhood of $[x]_{\sim}$ in the quotient space $X/\!\sim$.

Otherwise, there would exist an infinite sequence of points  $[x_k]_{\sim}$ in $V\setminus \pi(Q)$ such that $\lim\limits_{k\rightarrow\infty}[x_k]_{\sim}= [x]_{\sim}$ under the quotient topology. Since $U=\pi^{-1}(V)$, every $x_k$ belongs to $U$. Let $Q_k$ be the component of $U$ that contains $x_k$. Here we have $Q_k\cap Q=\emptyset$. And, by the above mentioned boundary bumping theorem, we also have $\overline{Q_k}\setminus U\ne\emptyset$.

Now, choose points $y_k\in [x_k]_{\sim}$ for every $k\ge1$ such that $\{y_k\}$ has a limit point $y$. Here, we certainly have $[y_k]_{\sim}=[x_k]_{\sim}$ and $[y]_{\sim}=[x]_{\sim}$.
By coming to an appropriate subsequence, we may assume that $\lim\limits_{k\rightarrow\infty}y_k=y$ and that  $\lim\limits_{k\rightarrow\infty}\overline{Q_k}=M$ under Hausdorff metric. Then $M$ is a continuum with $y\in M\subset Q$ and $M\setminus U\ne\emptyset$, indicating that the fiber $F_y^*$ contains $M$, hence intersects $X\setminus U$.  In other words, $F_y^*\nsubseteq U$, which contradicts the inclusions $y\in Q\subset U$ and $F_y^*\subset [y]_{\sim}=[x]_\sim\subset U$.
\end{proof}

\section{How fibers are changed under continuous maps}\label{fiber-basics}

In this section, we discuss how fibers  are changed under continuous maps. As a special application, we may compare the dynamics of a polynomial $f_c(z)=z^n+c$ on its Julia set $J_c$, the expansion $z\mapsto z^d$ on unit circle, and an induced map $\tilde{f}_c$ on the quotient $J_c/\!\sim$.

Let $X, Y\subset\bbC$ be continua and $x\in X$ a point. The first primary observation is that $f(F_x^*)\subset F_{f(x)}^*$ for any finite-to-one continuous surjection $f: X\rightarrow Y$.

Indeed, for any $y\ne x$ in the fiber $F_x^*$ and any finite set $C\subset Y$ that is disjoint from $\{f(x),f(y)\}$, we can see that $f^{-1}(C)$ is a finite set disjoint from $\{x,y\}$. Since $y\in F_x^*$ there exists no separation $X\setminus f^{-1}(C)=A\cup B$ with $x\in A, y\in B$; therefore, there exists no separation $Y\setminus C=P\cup Q$ with $f(x)\in P, f(y)\in Q$.  This certifies that $f(y)\in F_{f(x)}^*$.

By the above inclusion $f(F_x^*)\subset F_{f(x)}^*$ we further have $f(X_0)\subset Y_0$. Here $X_0$ is the union of all the nontrivial fibers $F_x^*$ in $X$, and $Y_0$ the union of those in $Y$. It follows that $f([x]_\sim)\subset[f(x)]_\sim$. Therefore, the correspondence $[x]_\sim\xrightarrow{\hspace{0.2cm}\tilde{f}\hspace{0.2cm}}[f(x)]_\sim$ gives a well defined map  $\tilde{f}: X/\!\sim\ \rightarrow Y/\!\sim$ that satisfies the following commutative diagram, in which each downward arrow $\downarrow$ indicates the natural projection $\pi$ from a space onto its quotient.
\[\begin{array}{ccc} X&\xrightarrow{\hspace{1cm}f\hspace{1cm}}&Y\\ \downarrow&&\downarrow\
\\ X/\!\sim&\xrightarrow{\hspace{1cm}\tilde{f}\hspace{1cm}}&Y/\!\sim
\end{array}\]
Given an open set $U\subset Y/\!\sim$, we can use the definition of quotient topology to infer that $V:=\tilde{f}^{-1}(U)$ is open in $X/\!\sim$ whenever $\pi^{-1}(V)$ is open in $X$. On the other hand, the above diagram ensures that $\pi^{-1}(V)=f^{-1}\left(\pi^{-1}(U)\right)$, which is an open set of $X$, by continuity of $f$ and $\pi$.

The above arguments lead us to a useful result for the study of dynamics on Julia sets.

\begin{theo}\label{dynamic}
Let $X, Y\subset\bbC$ be continua. Let the relation $\sim$ be defined as in Theorem \ref{main-3}. If $f: X\rightarrow Y$ is continuous, surjective and finite-to-one then $\tilde{f}([x]_\sim):=[f(x)]_\sim$ defines a continuous map with $\pi\circ f=\tilde{f}\circ\pi$.
\end{theo}

\begin{rema}
Every polynomial $f_c(z)=z^n+c$ restricted to its Julia set $J_c$ satisfies the conditions of Theorem \ref{dynamic}, if we assume that $J_c$ is connected; so the restricted system $f_c: J_c\rightarrow J_c$ has a factor system $\tilde{f}_c: J_c/\!\sim\rightarrow J_c/\!\sim$, whose underlying space is a locally connected continuum.
\end{rema}

Let $X\subset\bbC$ be an unshielded continuum and $U_\infty$ the unbounded component of $\bbC\setminus X$. Here, $X$ is unshielded provided that $X=\partial U_\infty$. Let $\bbD:=\{z\in\hat{\bbC}: |z|\le1\}$ be the unit closed disk. By Riemann Mapping Theorem, there exists a conformal isomorphism $\Phi: \hat{\bbC}\setminus\bbD\rightarrow U_\infty$  that fixes $\infty$ and has positive derivative at $\infty$. The prime end theory \cite{CP02, UrsellYoung51} builds a correspondence between an angle $\theta\in S^1:=\partial\bbD$ and a continuum
\[\displaystyle Imp(\theta):=\left\{w\in X: \ \exists\ z_n\in\bbD\ \text{with}\ z_n\rightarrow e^{{\bf i}\theta}, \lim\limits_{n\rightarrow\infty}\Phi(z_n)=w\right\}\]
We call $Imp(\theta)$ the {\em impression of $\theta$}. By \cite[p.173, Theorem 9.4]{CL66}, we may fix a simple open arc ${\mathcal R}_\theta$ in $\bbC\setminus\bbD$ landing at the point $e^{{\bf i}\theta}$ such that $\overline{\Phi({\mathcal R}_\theta)}\cap X=Imp(\theta)$.

We will connect impressions to fibers. Before that, we obtain a useful lemma concerning good cuts of an unshielded continuum $X$ on the plane. Here a good cut of $X$ is a simple closed curve that intersects $X$ at a finite subset (see Remark~\ref{modified}).

\begin{lemm}\label{optimal-cut}
Let $X\subset\bbC$ be an unshielded continuum and $U_\infty$ the unbounded component of $\bbC\setminus X$. Let $x$ and $y$ be two points on $X$ separated by a good cut of $X$. Then we can find a good cut separating $x$ from $y$ that intersects $U_\infty$ at an open arc.
\end{lemm}
\begin{proof}
Since each of the two components of $\bbC\setminus\gamma$ intersects $\{x,y\}$, we have $\gamma\cap U_\infty\ne\emptyset$. Since $\gamma\cap X$ is a finite set, the difference $\gamma\setminus X$ has finitely many components. Let $\gamma_1,\ldots,\gamma_k$ be the components of $\gamma\setminus X$ that lie in $U_\infty$. Let $\alpha_i=\Phi^{-1}(\gamma_i)$ be the pre-images of $\gamma_i$ under $\Phi$. Then every $\alpha_i$ is a simple open arc in $\{z: |z|>1\}$ whose end points $a_i, b_i$ are located on the unit circle; and all those open arcs $\alpha_1,\ldots,\alpha_k$ are pairwise disjoint.

If $k\geq 2$, rename the arcs $\alpha_2,\ldots,\alpha_k$ so that we can find an open arc $\beta\subset(\bbC\setminus\bbD)$ disjoint from $\bigcup_{i=1}^k\alpha_i$ that connects a point $a$ on $\alpha_1$ to a point $b$ on $\alpha_2$. Then $\gamma\cup\Phi(\beta)$ is a $\Theta$-curve separating $x$ from $y$ (see~\cite[Part B, Section VI]{Whyburn79} for a definition of $\Theta$-curve). Let $J_1$ and $J_2$ denote the two components of $\gamma\setminus\overline{\Phi(\beta)}=\gamma\setminus\{\Phi(a),\Phi(b)\}$. Then $J_1\cup\overline{\Phi(\beta)}$ and $J_2\cup\overline{\Phi(\beta)}$ are both good cuts of $X$. One of them, denoted by $\gamma'$, separates $x$ from $y$~\cite[$\Theta$-curve theorem, p.123]{Whyburn79}. By construction, this new good cut intersects $U_\infty$ at $k'$ open arcs for some $1\leq k'\leq k-1$. For relative locations of $J_1,J_2$ and $\Phi(\beta)$ in $\hat{\bbC}$, we refer to Figure~\ref{theta} in which $\gamma$ is represented as a circular circle, although a general good cut is usually not a circular circle.
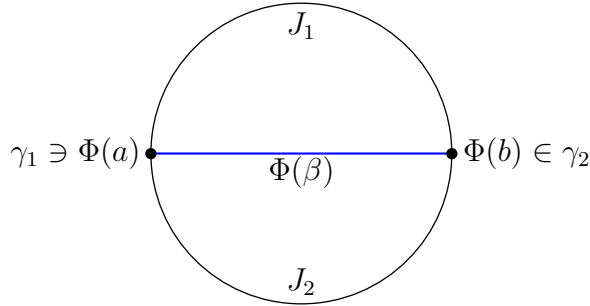
\begin{figure}[ht]
\begin{center}
\begin{pspicture}(4,4)(0,0.3)
\linethickness{0.5pt}

\rput(2,2){\circle{4}}
\rput(0,2){\psline{-,linecolor=blue}(0,0)(4,0)}
\rput(2,1.75){$\Phi(\beta)$}  \rput(2,3.7){$J_1$}  \rput(2,0.3){$J_2$}
\multiput(0,2)(4,0){2}{\circle*{0.15}}  \rput(-1,2){$\gamma_1\ni\Phi(a)$} \rput(5,2){$\Phi(b)\in\gamma_2$}
\end{pspicture}
\end{center}
\caption{The $\Theta$-curve together with the arcs $J_1$, $J_2$, and $\Phi(\beta)$.}\label{theta}
\end{figure}
If $k'\geq 2$, we may use the same argument on $\gamma'$ and obtain a good cut $\gamma''$, that separates $x$ from $y$ and that intersects $U_\infty$ at $k''$ open arcs for some $1\leq k''\leq  k-2$. Repeating this procedure for at most $k-1$ times, we will obtain a good cut separating $x$ from $y$ that intersects $U_\infty$ at a single open arc.
\end{proof}

\begin{theo}\label{impression}
Let $X\subset\bbC$ be an unshielded continuum. Then every impression $Imp(\theta)$ is contained in a fiber $F_w^*$ for some $w\in Imp(\theta)$.
\end{theo}
\begin{proof}
Suppose that a point $y\ne x$ on $Imp(\theta)$ is separated from $x$ in $X$ by a finite set. By Proposition \ref{key}, we can find a good cut $\gamma$ separating $x$ from $y$. By Lemma \ref{optimal-cut}, we may assume that $\gamma\cap U_\infty$ is an open arc $\gamma_1$. Let $a$ and $b$ be the two end points of $\alpha_1=\Phi^{-1}(\gamma_1)$, an open arc in $\bbC\setminus\bbD$.

Fix an open arc ${\mathcal R}_\theta$ in $\bbC\setminus\bbD$ landing at the point $e^{{\bf i}\theta}$ such that $\overline{\Phi({\mathcal R}_\theta)}\cap X=Imp(\theta)$. We note that $e^{{\bf i}\theta}\in \{a,b\}$. Otherwise, there is a number $r>1$ such that $\mathcal{R}_\theta\cap\{z: |z|<r\}$ lies in the component of
\[(\bbC\setminus\bbD)\setminus(\{a,b\}\cup \alpha_1)\qquad (\text{difference\ of}\ \bbC\setminus\bbD\ \text{and}\ \{a,b\}\cup \alpha_1)\]
whose closure contains $e^{{\bf i}\theta}$. From this we see that
$\Phi(\mathcal{R}_\theta\cap\{z: |z|<r\})$
is disjoint from $\gamma$ and is entirely contained in one of the two components of $\bbC\setminus\gamma$, which contain $x$ and $y$ respectively. Therefore,
\[\overline{\Phi(\mathcal{R}_\theta\cap\{z: |z|<r\})}\]
hence its subset $Imp(\theta)$ cannot contain  $x$ and $y$ at the same time. This contradicts the assumption that $x,y\in Imp(\theta)$.

Now we will lose no generality by assuming that $e^{{\bf i}\theta}=a$. Then $\Phi({\mathcal R}_\theta)$ intersects $\gamma_1$ infinitely many times, since $\overline{\Phi({\mathcal R}_\theta)}\setminus\Phi({\mathcal R}_\theta)$ contains $\{x,y\}$. This implies that $a$ is the landing point of $\mathcal{R}_\theta\subset(\bbC\setminus\bbD)$.

Let $w=\lim_{z\rightarrow a}\left.\Phi\right|_{\alpha_1}(z)$. Then $\{x,y,w\}\subset Imp(\theta)$, and the proof will be completed if we can verify that $Imp(\theta)\subset F_w^*$.

Suppose there is a point $w_1\in Imp(\theta)$ that is not in $F_w^*$. By Lemma \ref{optimal-cut} we may find a good cut $\gamma'$ separating $w$ from $w_1$  that intersects $U_\infty$ at an open arc $\gamma_1'$. Let $\alpha_1'=\Phi^{-1}(\gamma_1')$. Let $I$  be the component of $S^1\setminus\overline{\alpha_1'}$ that contains $a$. Since $w\notin\gamma'$, the closure $\overline{\alpha_1'}$ does not contain the point $a$. Therefore, $\mathcal{R}_\theta\cap\{z: |z|<r_1\}$ is disjoint from $\alpha_1'$ for some $r_1>1$. For such an $r_1$, the image $\Phi(\mathcal{R}_\theta\cap\{z: |z|<r_1\})$   is disjoint from $\gamma'$. On the other hand, the good cut $\gamma'$ separates $w$ from $w_1$. Therefore, the closure of $\Phi(\mathcal{R}_\theta\cap\{z: |z|<r_1\})$  hence its subset $Imp(\theta)$ does not contain  the two points $w$ and $w_1$ at the same time. This is a contradiction.
\end{proof}

\begin{rema}\label{z^d-factor}
Let $J_c$ be the connected Julia set of a polynomial.
The equivalence classes $[x]_\sim$ obtained in this paper determine an upper semi-continuous decomposition of $J_c$, such that the quotient space is a locally connected continuum. Theorem \ref{impression} says that the impression of any prime end is entirely contained in a single class $[x]_\sim$. Therefore, the finest decomposition mentioned in \cite[Theorem 1]{BCO11} is finer than $\{[x]_\sim: x\in J_c\}$. Currently it is not clear whether these two decompositions just coincide. This is proposed as an open question in Problem \ref{finest}.
\end{rema}

\section{Facts and Examples}\label{examples}

In this section, we give several examples to demonstrate the difference between (1) separating and  cutting sets, (2) the fiber $F_x^*$ and the class $[x]_\sim$, (3) a continuum $X\subset\bbC$ and the quotient space $X/\!\sim$ for specific choices of $X$. We also construct an infinite sequence of continua which have scales of any $k\ge 2$ and even up to $\infty$, although the quotient of each of those continua is always homeomorphic with the unit interval $[0,1]$.

\begin{exam}[{\bf Separating Sets and Cutting sets}]\label{cutting-2}

For a set $M\subset\bbC$,
a set $C\subset M$ is said to \emph{separate} or to be a \emph{separating set between} two points $a,b\subset M$ if there is a separation $M\setminus C=P\cup Q$ satisfying $a\in P, b\in Q$;
and a subset $C\subset M$ is called a {\em cutting set between two points $a,b\in M$} if $\{a,b\}\subset(X\setminus C)$ and if the component of $X\setminus C$ containing $a$ does not contain $b$ \cite[p.188,$\S$47.VIII]{Kuratowski68}.

Let $L_1$ be the segment between the points $(2,1)$ and $(2,0)$ on the plane, $Q_1$ the one between $(-2,0)$ and $c=(0,\frac{1}{2})$, and $P_1$ the broken line connecting $(2,0)$ to $(-2,0)$ through $(0,-1)$, as shown in Figure \ref{cutting-separating}.
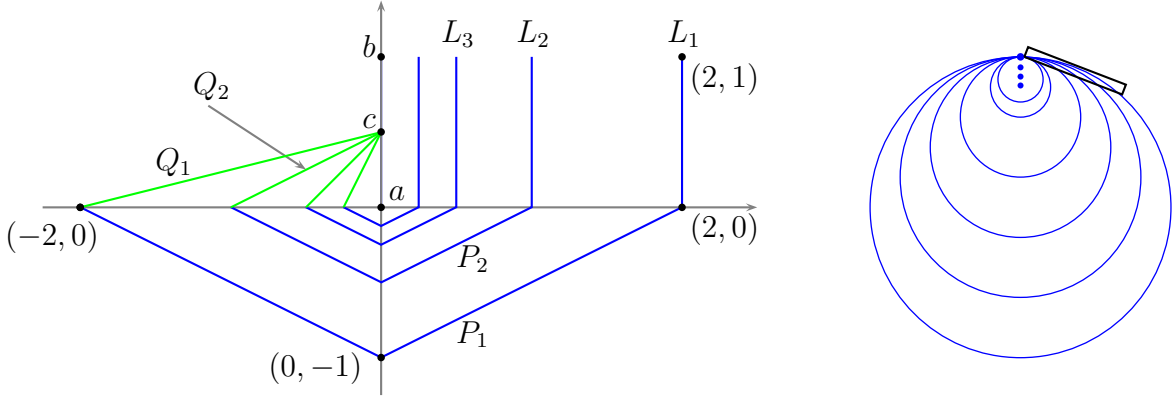
\begin{figure}[ht]
\begin{pspicture}(15,4.35)(-0.5,-0.25)
\linethickness{0.5pt}

\put(4.0,2.0){\psline{-,linecolor=blue}(0,0)(0,2)}
\put(1.5,3.5){\psline{->,linecolor=gray}(0.2,-0.15)(1.5,-1.0)}

\put(-0.5,2){\psline{->,linecolor=gray}(0.0,0.0)(9.5,0)}
\put(4.0,-0.5){\psline{->,linecolor=gray}(0.0,0.0)(0.0,5.25)}

\put(8.0,4.0){\psline{-,linecolor=blue}(0,0.0)(0.0,-2.0)(-4,-4)(-8,-2)}
\put(0,2){\psline{-,linecolor=green}(0.0,0.0)(4,1)}

\put(6.0,4.0){\psline{-,linecolor=blue}(0.0,0.0)(0.0,-2.0)(-2,-3)(-4,-2)}
\put(2.0,2.0){\psline{-,linecolor=green}(0,0)(2,1)}

\put(5.0,4.0){\psline{-,linecolor=blue}(0.0,0.0)(0.0,-2.0)(-1,-2.5)(-2,-2)}
\put(3,2.0){\psline{-,linecolor=green}(0,0)(1,1)}

\put(4.5,4.0){\psline{-,linecolor=blue}(0.0,0.0)(0.0,-2.0)(-0.5,-2.25)(-1,-2)}
\put(3.5,2.0){\psline{-,linecolor=green}(0,0)(0.5,1)}

\put(4.0,2.0){\circle*{0.1}}
\put(4.0,4.0){\circle*{0.1}} \put(0,2.0){\circle*{0.1}} \put(4.0,3.0){\circle*{0.1}}

\put(3.75,3.05){$c$} \put(3.75,4.0){$b$} \put(1.5,3.5){$Q_2$}  \put(4.1,2.1){$a$}

\put(7.8,4.2){$L_1$} \put(5,0.2){$P_1$} \put(1.0,2.5){$Q_1$}
\put(8.0,2.0){\circle*{0.1}} \put(8.1,1.6){$(2,0)$}
\put(8.0,4.0){\circle*{0.1}} \put(8.1,3.6){$(2,1)$}
\put(4.0,0){\circle*{0.1}} \put(2.5,-0.25){$(0,-1)$}
\put(0,2){\circle*{0.1}} \put(-1,1.5){$(-2,0)$}

\put(5.8,4.2){$L_2$}  \put(4.8,4.2){$L_3$}
\put(5,1.2){$P_2$}

\textcolor{blue}{
\put(12.5,4){\circle*{0.1}}
\put(12.5,2){\circle{4}} 
\put(12.5,2.4){\circle{3.2}} 
\put(12.5,2.8){\circle{2.4}}
\put(12.5,3.2){\circle{1.6}}
\put(12.5,3.6){\circle{0.8}}
\put(12.5,3.7){\circle{0.6}}
\multiput(12.5,3.62)(0,0.12){3}{\circle*{0.08}} }

\put(12.55,4){\psline{-,linecolor=black}(0,0)(1.3,-0.5)(1.35,-0.37)(0.05,0.13)(0,0)}

\end{pspicture}
\caption{The continuum $X$ and its quotient as a Hawaiian earring minus an open rectangle.}\label{cutting-separating}
\end{figure}
Define $(x_1,x_2)\xrightarrow{\hspace{0.2cm}f\hspace{0.2cm}}(\frac{1}{2}x_1,\frac{1}{2}x_2)$ and  $(x_1,x_2)\xrightarrow{\hspace{0.2cm}g\hspace{0.2cm}}(\frac{1}{2}x_1,x_2)$.
For any $k\ge1$, let $L_{k+1}=g(L_k)$ and $Q_{k+1}=g(Q_k)$; let $P_{k+1}=f(P_k)$. Let $B_k=L_k\cup P_k\cup Q_k$. Then $\{B_k: k\ge1\}$ is a sequence of broken lines converging to the segment $B$ between $a=(0,0)$ and $b=(0,1)$. Let $N=\left(\bigcup_kB_k\right)\bigcup B$. Then $N$ is a continuum, which is not locally connected at each point of $B$. Moreover, the singleton $\{c\}$ is a cutting set, but not a separating set, between the points $a$ and $b$. The only nontrivial fiber is $B=\{0\}\times[0,1]=F_x^*$ for each $x\in B$. So we have $\ell^*(N)=1$. Also, it follows that $[x]_\sim=B$ for all $x\in B$ and $[x]_\sim=\{x\}$ otherwise. In particular, the broken lines $B_k$ are still arcs in the quotient space but, under the metric of quotient space, their diameters converge to zero. Consequently, the quotient $N/\!\sim$ is topologically the difference of a Hawaiian earring with a full open rectangle. See the right part of Figure~\ref{cutting-separating}. In other words, the quotient space $N/\!\sim$ is homeomorphic with the quotient $X/\!\sim$ of Example~\ref{witch-broom}.

\end{exam}

\begin{exam}[{\bf The Witch's Broom}]\label{witch-broom}
Let $X$ be the witch's broom \cite[p.84, Figure 5.22]{Nadler92}. See Figure~\ref{broom}.
\begin{figure}[ht]
\begin{pspicture}(14,3.5)(-0.5,-0.25)
\linethickness{0.5pt}
\put(6.4,0){\psline{-,linecolor=blue,linewidth=0.5pt}(0.0,0.0)(-3.2,3.2)}
\put(6.4,0){\psline{-,linecolor=blue,linewidth=0.5pt}(0.0,0.0)(-3.2,1.6)}
\put(6.4,0){\psline{-,linecolor=blue,linewidth=0.5pt}(0.0,0.0)(-3.2,0.8)}
\put(6.4,0){\psline{-,linecolor=blue,linewidth=0.5pt}(0.0,0.0)(-3.2,0.4)}
\put(6.4,0){\psline{-,linecolor=blue,linewidth=0.5pt}(0.0,0.0)(-3.2,0.2)}
\put(6.4,0){\psline{-,linecolor=blue,linewidth=0.5pt}(0.0,0.0)(-3.2,0.1)}

\put(3.2,0){\psline{-,linecolor=blue,linewidth=0.5pt}(0.0,0.0)(-1.6,1.6)}
\put(3.2,0){\psline{-,linecolor=blue,linewidth=0.5pt}(0.0,0.0)(-1.6,0.8)}
\put(3.2,0){\psline{-,linecolor=blue,linewidth=0.5pt}(0.0,0.0)(-1.6,0.4)}
\put(3.2,0){\psline{-,linecolor=blue,linewidth=0.5pt}(0.0,0.0)(-1.6,0.2)}
\put(3.2,0){\psline{-,linecolor=blue,linewidth=0.5pt}(0.0,0.0)(-1.6,0.1)}

\put(1.6,0){\psline{-,linecolor=blue,linewidth=0.5pt}(0.0,0.0)(-0.8,0.8)}
\put(1.6,0){\psline{-,linecolor=blue,linewidth=0.5pt}(0.0,0.0)(-0.8,0.4)}
\put(1.6,0){\psline{-,linecolor=blue,linewidth=0.5pt}(0.0,0.0)(-0.8,0.2)}
\put(1.6,0){\psline{-,linecolor=blue,linewidth=0.5pt}(0.0,0.0)(-0.8,0.1)}

\put(0.8,0){\psline{-,linecolor=blue,linewidth=0.5pt}(0.0,0.0)(-0.4,0.4)}
\put(0.8,0){\psline{-,linecolor=blue,linewidth=0.5pt}(0.0,0.0)(-0.4,0.2)}
\put(0.8,0){\psline{-,linecolor=blue,linewidth=0.5pt}(0.0,0.0)(-0.4,0.1)}
\put(0.8,0){\psline{-,linecolor=blue,linewidth=0.5pt}(0.0,0.0)(-0.4,0.05)}

\put(0.4,0){\psline{-,linecolor=blue,linewidth=0.5pt}(0.0,0.0)(-0.2,0.2)}
\put(0.4,0){\psline{-,linecolor=blue,linewidth=0.5pt}(0.0,0.0)(-0.2,0.1)}

\put(0.2,0){\psline{-,linecolor=blue,linewidth=0.5pt}(0.0,0.0)(-0.1,0.1)}

\put(0.1,0){\psline{-,linecolor=blue,linewidth=0.5pt}(0.0,0.0)(-0.05,0.05)}

\put(0,0){\psline{-,linecolor=blue,linewidth=0.5pt}(0.0,0.0)(6.4,0)}
\put(6.4,0){\circle*{0.1}}  \put(3.2,0){\circle*{0.1}} \put(1.6,0){\circle*{0.1}}
\put(0.8,0){\circle*{0.1}} \put(0.4,0){\circle*{0.1}}  \put(0,0){\circle*{0.1}}

\textcolor{red}{\put(12,1.6){\circle{4}} }
\put(12,-0.4){\psline{-,linecolor=blue,linewidth=0.5pt}(0.0,0.0)(0,4)}
\put(12,-0.4){\psline{-,linecolor=blue,linewidth=0.5pt}(0.0,0.0)(0.5,3.93)}
\put(12,-0.4){\psline{-,linecolor=blue,linewidth=0.5pt}(0.0,0.0)(0.9,3.78)}
\put(12,-0.4){\psline{-,linecolor=blue,linewidth=0.5pt}(0.0,0.0)(1.3,3.51)}
\put(12,-0.4){\psline{-,linecolor=blue,linewidth=0.5pt}(0.0,0.0)(1.7,3.05)}
\put(12,-0.4){\psline{-,linecolor=blue,linewidth=0.5pt}(0.0,0.0)(1.91,2.6)}
\put(12,-0.4){\psline{-,linecolor=blue,linewidth=0.5pt}(0.0,0.0)(2.0,2.0)}
\put(12,-0.4){\psline{-,linecolor=blue,linewidth=0.5pt}(0.0,0.0)(1.91,1.4)}
\put(12,-0.4){\psline{-,linecolor=blue,linewidth=0.5pt}(0.0,0.0)(1.7,0.95)}
\put(12,-0.4){\psline{-,linecolor=blue,linewidth=0.5pt}(0.0,0.0)(1.3,0.49)}
\multiput(12.3,-0.33)(0.1,0.03){6}{\circle*{0.05}}

\end{pspicture}
\caption{An intuitive depiction of the witch's broom and its quotient space.}\label{broom}
\end{figure}
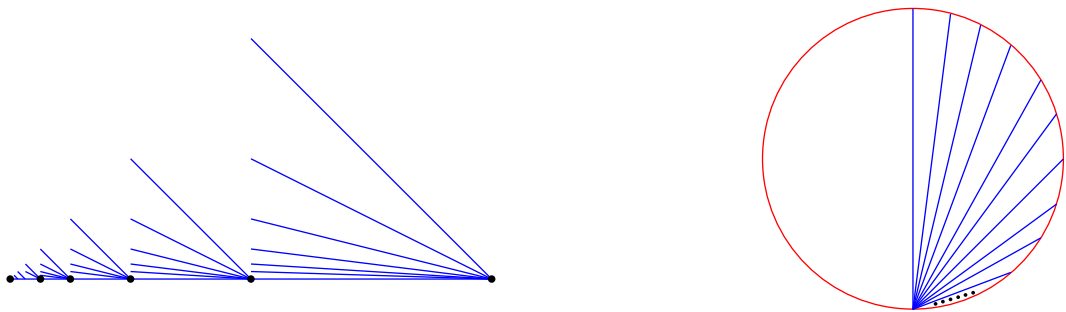
More precisely, let $A_0:=[\frac{1}{2},1]\times\{0\}$; let $A_k$ be the segment connecting $(1,0)$ to $(\frac{1}{2},2^{-k})$ for $k\ge0$. Then $\displaystyle A=\bigcup_{k\ge0}A_k$ is a continuum (an infinite broom) which is locally connected everywhere but at the points on $[\frac{1}{2},1)\times\{0\}$. Let $g(x)=\frac{1}{2}x$ be a similarity contraction on $\bbR^2$. Let
\[X=\{(0,0)\}\cup A\cup f(A)\cup f^2(A)\cup\cdots\cup f^n(A)\cup\cdots\cdots.\]
The continuum $X$ is called the {\em Witch's Broom}. Consider the fibers of $X$, we have $F_x^*=\{x\}$ for each $x$ in $X\cap\{(x_1,x_2): x_2>0\}$ and for $x=(0,0)$. The nontrivial fibers include: $F_{(1,0)}^*=[\frac{1}{2},1]\times\{0\}$, $F_{(2^{-k},0)}^*=[2^{-k-1},2^{-k+1}]\times\{0\}\quad(k\ge1)$,
and
\[ F^*_{(x_1,0)}=[2^{-k},2^{-k+1}]\times\{0\}\quad (2^{-k}<x_1<2^{-k+1}, k\ge1).\]
Consequently, $[x]_\sim=\{x\}$ for each $x$ in $X\cap\{(x_1,x_2): x_2>0\}$, while $[x]_\sim=[0,1]\times\{0\}$ for $x\in[0,1]\times\{0\}$. See the right part of Figure \ref{broom} for a depiction of the quotient $X/\!\sim$.
\end{exam}

\begin{exam}[{\bf Witch's Double Broom}]\label{relations}
Let $X$ be the witch's broom. We call the union $Y$ of $X$ with a translated copy $X+(-1,0)$ the {\em witch's double broom} (see Figure \ref{double-broom}).
\begin{figure}[ht]
\begin{pspicture}(14,3.2)(-0.5,-0.3)
\linethickness{0.5pt}
\put(6.4,0){\psline{-,linecolor=blue,linewidth=0.5pt}(0.0,0.0)(-3.2,3.2)}
\put(6.4,0){\psline{-,linecolor=blue,linewidth=0.5pt}(0.0,0.0)(-3.2,1.6)}
\put(6.4,0){\psline{-,linecolor=blue,linewidth=0.5pt}(0.0,0.0)(-3.2,0.8)}
\put(6.4,0){\psline{-,linecolor=blue,linewidth=0.5pt}(0.0,0.0)(-3.2,0.4)}
\put(6.4,0){\psline{-,linecolor=blue,linewidth=0.5pt}(0.0,0.0)(-3.2,0.2)}
\put(6.4,0){\psline{-,linecolor=blue,linewidth=0.5pt}(0.0,0.0)(-3.2,0.1)}

\put(3.2,0){\psline{-,linecolor=blue,linewidth=0.5pt}(0.0,0.0)(-1.6,1.6)}
\put(3.2,0){\psline{-,linecolor=blue,linewidth=0.5pt}(0.0,0.0)(-1.6,0.8)}
\put(3.2,0){\psline{-,linecolor=blue,linewidth=0.5pt}(0.0,0.0)(-1.6,0.4)}
\put(3.2,0){\psline{-,linecolor=blue,linewidth=0.5pt}(0.0,0.0)(-1.6,0.2)}
\put(3.2,0){\psline{-,linecolor=blue,linewidth=0.5pt}(0.0,0.0)(-1.6,0.1)}

\put(1.6,0){\psline{-,linecolor=blue,linewidth=0.5pt}(0.0,0.0)(-0.8,0.8)}
\put(1.6,0){\psline{-,linecolor=blue,linewidth=0.5pt}(0.0,0.0)(-0.8,0.4)}
\put(1.6,0){\psline{-,linecolor=blue,linewidth=0.5pt}(0.0,0.0)(-0.8,0.2)}
\put(1.6,0){\psline{-,linecolor=blue,linewidth=0.5pt}(0.0,0.0)(-0.8,0.1)}

\put(0.8,0){\psline{-,linecolor=blue,linewidth=0.5pt}(0.0,0.0)(-0.4,0.4)}
\put(0.8,0){\psline{-,linecolor=blue,linewidth=0.5pt}(0.0,0.0)(-0.4,0.2)}
\put(0.8,0){\psline{-,linecolor=blue,linewidth=0.5pt}(0.0,0.0)(-0.4,0.1)}
\put(0.8,0){\psline{-,linecolor=blue,linewidth=0.5pt}(0.0,0.0)(-0.4,0.05)}

\put(0.4,0){\psline{-,linecolor=blue,linewidth=0.5pt}(0.0,0.0)(-0.2,0.2)}
\put(0.4,0){\psline{-,linecolor=blue,linewidth=0.5pt}(0.0,0.0)(-0.2,0.1)}

\put(0.2,0){\psline{-,linecolor=blue,linewidth=0.5pt}(0.0,0.0)(-0.1,0.1)}

\put(0.1,0){\psline{-,linecolor=blue,linewidth=0.5pt}(0.0,0.0)(-0.05,0.05)}

\put(0,0){\psline{-,linecolor=blue,linewidth=0.5pt}(0.0,0.0)(6.4,0)}
\put(6.4,0){\circle*{0.1}}  \put(3.2,0){\circle*{0.1}} \put(1.6,0){\circle*{0.1}}
\put(0.8,0){\circle*{0.1}} \put(0.4,0){\circle*{0.1}}  \put(0,0){\circle*{0.1}}

\put(12.8,0){\psline{-,linecolor=blue,linewidth=0.5pt}(0.0,0.0)(-3.2,3.2)}
\put(12.8,0){\psline{-,linecolor=blue,linewidth=0.5pt}(0.0,0.0)(-3.2,1.6)}
\put(12.8,0){\psline{-,linecolor=blue,linewidth=0.5pt}(0.0,0.0)(-3.2,0.8)}
\put(12.8,0){\psline{-,linecolor=blue,linewidth=0.5pt}(0.0,0.0)(-3.2,0.4)}
\put(12.8,0){\psline{-,linecolor=blue,linewidth=0.5pt}(0.0,0.0)(-3.2,0.2)}
\put(12.8,0){\psline{-,linecolor=blue,linewidth=0.5pt}(0.0,0.0)(-3.2,0.1)}

\put(9.6,0){\psline{-,linecolor=blue,linewidth=0.5pt}(0.0,0.0)(-1.6,1.6)}
\put(9.6,0){\psline{-,linecolor=blue,linewidth=0.5pt}(0.0,0.0)(-1.6,0.8)}
\put(9.6,0){\psline{-,linecolor=blue,linewidth=0.5pt}(0.0,0.0)(-1.6,0.4)}
\put(9.6,0){\psline{-,linecolor=blue,linewidth=0.5pt}(0.0,0.0)(-1.6,0.2)}
\put(9.6,0){\psline{-,linecolor=blue,linewidth=0.5pt}(0.0,0.0)(-1.6,0.1)}

\put(8,0){\psline{-,linecolor=blue,linewidth=0.5pt}(0.0,0.0)(-0.8,0.8)}
\put(8,0){\psline{-,linecolor=blue,linewidth=0.5pt}(0.0,0.0)(-0.8,0.4)}
\put(8,0){\psline{-,linecolor=blue,linewidth=0.5pt}(0.0,0.0)(-0.8,0.2)}
\put(8,0){\psline{-,linecolor=blue,linewidth=0.5pt}(0.0,0.0)(-0.8,0.1)}

\put(7.2,0){\psline{-,linecolor=blue,linewidth=0.5pt}(0.0,0.0)(-0.4,0.4)}
\put(7.2,0){\psline{-,linecolor=blue,linewidth=0.5pt}(0.0,0.0)(-0.4,0.2)}
\put(7.2,0){\psline{-,linecolor=blue,linewidth=0.5pt}(0.0,0.0)(-0.4,0.1)}
\put(7.2,0){\psline{-,linecolor=blue,linewidth=0.5pt}(0.0,0.0)(-0.4,0.05)}

\put(6.8,0){\psline{-,linecolor=blue,linewidth=0.5pt}(0.0,0.0)(-0.2,0.2)}
\put(6.8,0){\psline{-,linecolor=blue,linewidth=0.5pt}(0.0,0.0)(-0.2,0.1)}

\put(6.6,0){\psline{-,linecolor=blue,linewidth=0.5pt}(0.0,0.0)(-0.1,0.1)}

\put(6.5,0){\psline{-,linecolor=blue,linewidth=0.5pt}(0.0,0.0)(-0.05,0.05)}

\put(6.4,0){\psline{-,linecolor=blue,linewidth=0.5pt}(0.0,0.0)(6.4,0)}
\put(12.8,0){\circle*{0.1}}  \put(9.6,0){\circle*{0.1}} \put(8,0){\circle*{0.1}}
\put(7.2,0){\circle*{0.1}} \put(6.8,0){\circle*{0.1}}  \put(6.4,0){\circle*{0.1}}

\put(-0.6,-0.5){$(-1,0)$}  \put(5.9,-0.5){$(0,0)$}  \put(12.3,-0.5){$(1,0)$}
\end{pspicture}
\caption{Relative locations of the points $(\pm1,0)$ and $(0,0)$ in witch's double  broom.}\label{double-broom}
\end{figure}
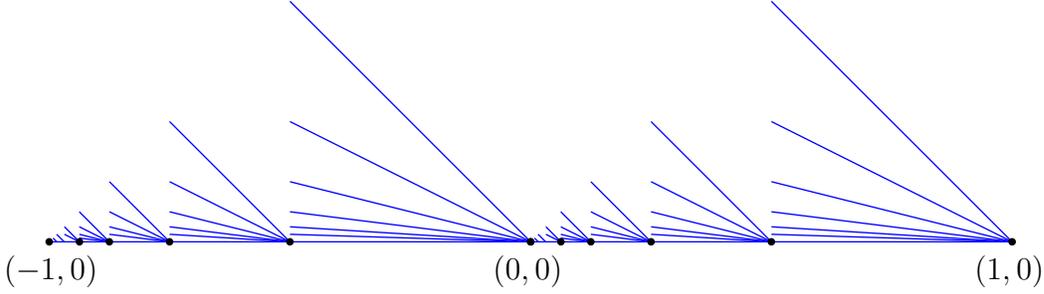
Define $x\approx y$ if there exist points $x_1=x,$ $x_2,\ldots, x_n=y$ in $Y$ such that $x_i\in F^*_{x_{i-1}}$. Then $\approx$ is an equivalence and is not closed. Its closure $\approx^*$ is not transitive, since we have $(-1,0)\approx^*(0,0)$ and $(0,0)\approx^*(1,0)$,  but $(-1,0)$ is not related to $(1,0)$ under $\approx^*$.
\end{exam}

\begin{exam}[{\bf Cantor's Teepee}]\label{cantor-teepee}
Let $X$ be Cantor's Teepee~\cite[p.145]{SS-1995}. See Figure \ref{teepee}. Then the fiber $F_p^*=X$; and for every other point $x$, $F_x^*$ is exactly the line segment on $X$ that crosses $x$ and $p$. Therefore, $\ell^*(X)=1$. Moreover, $[x]_\sim=X$ for every $x$, hence the quotient is a single point. In this case, we also say that $X$ is collapsed to a point.
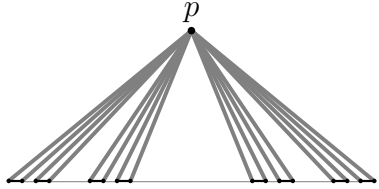
\begin{figure}[ht]
\begin{center}
\begin{pspicture}(6.0,2.25)(0.0,0.25)
\linethickness{1pt}
\rput(3.0,2.25){$p$}
\rput(0.57,-0.01){\psline{-,linecolor=gray,linewidth=0.382pt}(0,0)(4.86,0)}

\rput(3,2.0){\psline{-,linecolor=gray,linewidth=1.5pt}(-2.43,-2.0)(0,0)}
\rput(3,2.0){\psline{-,linecolor=gray,linewidth=1.5pt}(0,0)(2.43,-2.0)}
\rput(3,2.0){\psline{-,linecolor=gray,linewidth=1.5pt}(-0.81,-2.0)(0,0)}
\rput(3,2.0){\psline{-,linecolor=gray,linewidth=1.5pt}(0,0)(0.81,-2.0)}

\rput(3,2.0){\psline{-,linecolor=gray,linewidth=1.5pt}(-1.89,-2.0)(0,0)}
\rput(3,2.0){\psline{-,linecolor=gray,linewidth=1.5pt}(0,0)(1.89,-2.0)}
\rput(3,2.0){\psline{-,linecolor=gray,linewidth=1.5pt}(-1.35,-2.0)(0,0)}
\rput(3,2.0){\psline{-,linecolor=gray,linewidth=1.5pt}(0,0)(1.35,-2.0)}

\rput(3,2.0){\psline{-,linecolor=gray,linewidth=1.5pt}(-2.25,-2.0)(0,0)}
\rput(3,2.0){\psline{-,linecolor=gray,linewidth=1.5pt}(0,0)(2.25,-2.0)}
\rput(3,2.0){\psline{-,linecolor=gray,linewidth=1.5pt}(-2.07,-2.0)(0,0)}
\rput(3,2.0){\psline{-,linecolor=gray,linewidth=1.5pt}(0,0)(2.07,-2.0)}

\rput(3,2.0){\psline{-,linecolor=gray,linewidth=1.5pt}(-1.17,-2.0)(0,0)}
\rput(3,2.0){\psline{-,linecolor=gray,linewidth=1.5pt}(0,0)(1.17,-2.0)}
\rput(3,2.0){\psline{-,linecolor=gray,linewidth=1.5pt}(-0.99,-2.0)(0,0)}
\rput(3,2.0){\psline{-,linecolor=gray,linewidth=1.5pt}(0,0)(0.99,-2.0)}

\rput(3,2.0){\circle*{0.10}}

\rput(3,2.0){\psline{-,linecolor=black}(-2.43,-2.0)(-2.25,-2.0)}
\rput(3,2.0){\psline{-,linecolor=black}(-2.07,-2.0)(-1.89,-2.0)}
\rput(3,2.0){\psline{-,linecolor=black}(-1.35,-2.0)(-1.17,-2.0)}
\rput(3,2.0){\psline{-,linecolor=black}(-0.99,-2.0)(-0.81,-2.0)}

\rput(3,2.0){\psline{-,linecolor=black}(2.43,-2.0)(2.25,-2.0)}
\rput(3,2.0){\psline{-,linecolor=black}(2.07,-2.0)(1.89,-2.0)}
\rput(3,2.0){\psline{-,linecolor=black}(1.35,-2.0)(1.17,-2.0)}
\rput(3,2.0){\psline{-,linecolor=black}(0.99,-2.0)(0.81,-2.0)}

\multiput(0.57,0)(0.18,0){4}{\circle*{0.06}}
\multiput(1.65,0)(0.18,0){4}{\circle*{0.06}}
\multiput(3.81,0)(0.18,0){4}{\circle*{0.06}}
\multiput(4.89,0)(0.18,0){4}{\circle*{0.06}}
\end{pspicture}\end{center}
\caption{A simple representation of Cantor's Teepee.}\label{teepee}
\end{figure}
\end{exam}

\begin{exam}[{\bf Cantor's Comb}]\label{cantor-comb}
Let $\Kc\subset[0,1]$ be Cantor's ternary set. Let $X$ be the union of $\Kc\times[0,1]$ with $[0,1]\times\{1\}$. See Figure \ref{comb}. We call $X$ the Cantor comb. Then the fiber $F_x^*=\{x\}$ for every point on $X$ that is off $\Kc\times[0,1]$; and for every point $x$ on $\Kc\times[0,1]$, the fiber $F_x^*$ is exactly the vertical line segment on $\Kc\times[0,1]$ that contains $x$. Therefore, $\ell^*(X)=1$. Moreover, $[x]_\sim=F_x^*$ for every $x$, hence the quotient is homeomorphic to $[0,1]$. Here, we note that $X$ is locally connected at every point lying on the common part of $[0,1]\times\{1\}$ and $\Kc\times[0,1]$, although the fibers at those points are each a non-degenerate segment.
\begin{figure}[ht]
\begin{tabular}{ccc}
\begin{pspicture}(2.0,2.0)(-2,0)
\linethickness{1pt}

\multiput(0,0)(0.08,0){4}{\psline{-,linecolor=gray,linewidth=1.5pt}(0,0)(0,2.16)}
\multiput(0.48,0)(0.08,0){4}{\psline{-,linecolor=gray,linewidth=1.5pt}(0,0)(0,2.16)}
\multiput(1.44,0)(0.08,0){4}{\psline{-,linecolor=gray,linewidth=1.5pt}(0,0)(0,2.16)}
\multiput(1.92,0)(0.08,0){4}{\psline{-,linecolor=gray,linewidth=1.5pt}(0,0)(0,2.16)}

\rput(-0.025,2.16){\psline{-,linecolor=gray,linewidth=2pt}(0,0)(2.21,0)}

\end{pspicture}
&
\begin{pspicture}(2.0,2.0)(-2,0)
\linethickness{1pt}

\multiput(0,0)(0.08,0){4}{\psline{-,linecolor=gray,linewidth=1.5pt}(0,0)(0,2.16)}
\multiput(0.48,0)(0.08,0){4}{\psline{-,linecolor=gray,linewidth=1.5pt}(0,0)(0,2.16)}
\multiput(1.44,0)(0.08,0){4}{\psline{-,linecolor=gray,linewidth=1.5pt}(0,0)(0,2.16)}
\multiput(1.92,0)(0.08,0){4}{\psline{-,linecolor=gray,linewidth=1.5pt}(0,0)(0,2.16)}


\end{pspicture}
&
\begin{pspicture}(4.0,2.0)(0,0)
\linethickness{1pt}
\rput(2,2.16){\psline{-,linecolor=blue,linewidth=2pt}(0,0)(2.16,0)}
\end{pspicture}
\end{tabular}
\caption{Cantor's Comb, its nontrivial fibers, and the quotient $X/\!\sim$.}\label{comb}
\end{figure}
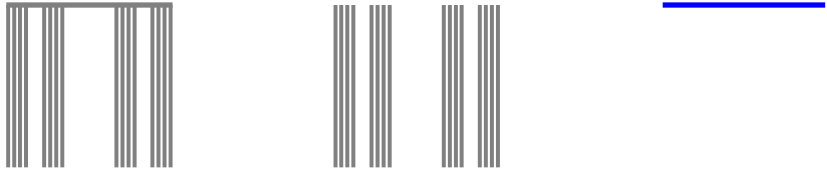
\end{exam}

\begin{exam}[{\bf More Combs}]\label{more-combs}
We use Cantor's ternary set $\Kc\subset[0,1]$ to construct a sequence of continua $\{X_k: k\ge1\}$, such that the scale $\ell^*(X_k)=k$ for all $k\ge1$. We also determine the fibers and compute the quotient spaces $X_k/\!\sim$. Let $X_1$ be the union of $X_1'=(\Kc+1)\times[0,2]$ with $[1,2]\times\{2\}$. Here $\displaystyle \Kc+1:=\left\{x_1+1: x_1\in\Kc\right\}$. Then $X_1$ is homeomorphic with Cantor's Comb defined in Example \ref{cantor-comb}. We have $\ell^*(X_1)=1$ and  that $X_1/\!\sim$ is homeomorphic with $[0,1]$. Let $X_2$ be the union of $X_1$ with $\displaystyle [0,1]\times(\Kc+1)$. See Figure \ref{combs}.
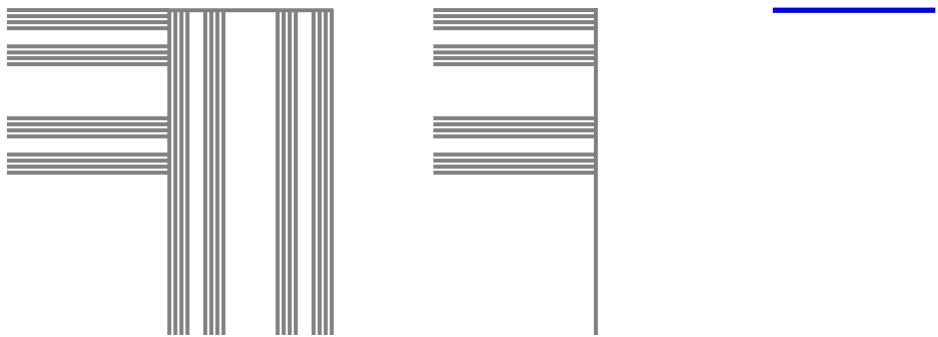
\begin{figure}[ht]
\begin{tabular}{ccc}
\begin{pspicture}(5.32,4.2)(0,0.25)
\linethickness{1.5pt}

\multiput(3.16,0)(0.08,0){4}{\psline{-,linecolor=gray,linewidth=1.5pt}(0,0)(0,4.32)}
\multiput(3.64,0)(0.08,0){4}{\psline{-,linecolor=gray,linewidth=1.5pt}(0,0)(0,4.32)}
\multiput(4.6,0)(0.08,0){4}{\psline{-,linecolor=gray,linewidth=1.5pt}(0,0)(0,4.32)}
\multiput(5.08,0)(0.08,0){4}{\psline{-,linecolor=gray,linewidth=1.5pt}(0,0)(0,4.32)}

\multiput(1,2.16)(0,0.08){4}{\psline{-,linecolor=gray,linewidth=1.5pt}(0,0)(2.16,0)}
\multiput(1,2.64)(0,0.08){4}{\psline{-,linecolor=gray,linewidth=1.5pt}(0,0)(2.16,0)}
\multiput(1,3.6)(0,0.08){4}{\psline{-,linecolor=gray,linewidth=1.5pt}(0,0)(2.16,0)}
\multiput(1,4.08)(0,0.08){4}{\psline{-,linecolor=gray,linewidth=1.5pt}(0,0)(2.16,0)}

\rput(3.16,4.32){\psline{-,linecolor=gray,linewidth=1.5pt}(0,0)(2.184,0)}

\end{pspicture}
&
\begin{pspicture}(3,2.0)(0,0.25)
\linethickness{1.5pt}

\multiput(1,2.16)(0,0.08){4}{\psline{-,linecolor=gray,linewidth=1.5pt}(0,0)(2.16,0)}
\multiput(1,2.64)(0,0.08){4}{\psline{-,linecolor=gray,linewidth=1.5pt}(0,0)(2.16,0)}
\multiput(1,3.6)(0,0.08){4}{\psline{-,linecolor=gray,linewidth=1.5pt}(0,0)(2.16,0)}
\multiput(1,4.08)(0,0.08){4}{\psline{-,linecolor=gray,linewidth=1.5pt}(0,0)(2.16,0)}

\rput(3.16,0){\psline{-,linecolor=gray,linewidth=1.5pt}(0,0)(0,4.35)}

\end{pspicture}
&
\begin{pspicture}(4,4.2)(0,0.25)
\linethickness{1pt}
\rput(2.16,4.32){\psline{-,linecolor=blue,linewidth=2pt}(0,0)(2.16,0)}
\end{pspicture}
\end{tabular}
\caption{A simple depiction of $X_2$, the largest fiber, and the quotient $X_2/\!\sim$.}\label{combs}
\end{figure}
Then the fiber of $X_2$ at the point $(1,2)\in X_2$ is $F_{(1,2)}^*=X_2\cap\{(x_1,x_2): x_1\le 1\}$, which will be referred to as the ``largest fiber'', since it is the fiber with the largest scale in $X_2$. See the central part of Figure~\ref{combs}. The other fibers are either a single point or a segment, of the form $\{(x_1,x_2): 0\le x_2\le 2\}$ for some $x_1\in \Kc+1$. Therefore, we have $\ell^*(X_2)=2$ and can check that the quotient $X_2/\!\sim$ is homeomorphic with $[0,1]$. Let $X_3$ be the union of $X_2$ with \[
\displaystyle \frac{X_1}{2}=\left\{\left(\frac{x_1}{2},\frac{x_2}{2}\right): (x_1,x_2)\in X_1\right\}.\]
Then the largest fiber of $X_3$ is exactly  $F_{(1,2)}^*=X_3\cap\{(x_1,x_2): x_1\le 1\}$, which is homeomorphic with $X_2$. Therefore, $\ell^*(X_3)=3$; moreover, $X_3/\!\sim$ is also homeomorphic with $[0,1]$. See upper part of Figure \ref{scale-4}.
\begin{figure}[ht]
\begin{tabular}{ccc}
\begin{pspicture}(5.32,4.2)(0,0.25)
\linethickness{1pt}

\multiput(3.16,0)(0.08,0){4}{\psline{-,linecolor=gray,linewidth=1pt}(0,0)(0,4.32)}
\multiput(3.64,0)(0.08,0){4}{\psline{-,linecolor=gray,linewidth=1pt}(0,0)(0,4.32)}
\multiput(4.6,0)(0.08,0){4}{\psline{-,linecolor=gray,linewidth=1pt}(0,0)(0,4.32)}
\multiput(5.08,0)(0.08,0){4}{\psline{-,linecolor=gray,linewidth=1pt}(0,0)(0,4.32)}

\multiput(1,2.16)(0,0.08){4}{\psline{-,linecolor=gray,linewidth=1pt}(0,0)(2.16,0)}
\multiput(1,2.64)(0,0.08){4}{\psline{-,linecolor=gray,linewidth=1pt}(0,0)(2.16,0)}
\multiput(1,3.6)(0,0.08){4}{\psline{-,linecolor=gray,linewidth=1pt}(0,0)(2.16,0)}
\multiput(1,4.08)(0,0.08){4}{\psline{-,linecolor=gray,linewidth=1pt}(0,0)(2.16,0)}

\multiput(2.08,0)(0.04,0){4}{\psline{-,linecolor=gray,linewidth=1pt}(0,0)(0,2.16)}
\multiput(2.32,0)(0.04,0){4}{\psline{-,linecolor=gray,linewidth=1pt}(0,0)(0,2.16)}
\multiput(2.8,0)(0.04,0){4}{\psline{-,linecolor=gray,linewidth=1pt}(0,0)(0,2.16)}
\multiput(3.04,0)(0.04,0){4}{\psline{-,linecolor=gray,linewidth=1pt}(0,0)(0,2.16)}

\rput(3.16,4.32){\psline{-,linecolor=gray,linewidth=1pt}(0,0)(2.175,0)}

\end{pspicture}
&
\begin{pspicture}(4,4.2)(0,0.25)
\linethickness{1pt}

\multiput(1,2.16)(0,0.08){4}{\psline{-,linecolor=gray,linewidth=1pt}(0,0)(2.16,0)}
\multiput(1,2.64)(0,0.08){4}{\psline{-,linecolor=gray,linewidth=1pt}(0,0)(2.16,0)}
\multiput(1,3.6)(0,0.08){4}{\psline{-,linecolor=gray,linewidth=1pt}(0,0)(2.16,0)}
\multiput(1,4.08)(0,0.08){4}{\psline{-,linecolor=gray,linewidth=1pt}(0,0)(2.16,0)}

\multiput(2.08,0)(0.04,0){4}{\psline{-,linecolor=gray,linewidth=1pt}(0,0)(0,2.16)}
\multiput(2.32,0)(0.04,0){4}{\psline{-,linecolor=gray,linewidth=1pt}(0,0)(0,2.16)}
\multiput(2.8,0)(0.04,0){4}{\psline{-,linecolor=gray,linewidth=1pt}(0,0)(0,2.16)}
\multiput(3.04,0)(0.04,0){4}{\psline{-,linecolor=gray,linewidth=1pt}(0,0)(0,2.16)}

\rput(3.16,0){\psline{-,linecolor=gray,linewidth=1pt}(0,0)(0,4.32)}

\end{pspicture}
&
\begin{pspicture}(4,4.2)(0,0.25)
\linethickness{1pt}
\rput(1.16,4.32){\psline{-,linecolor=blue,linewidth=2pt}(0,0)(2.16,0)}
\end{pspicture}
\\

\begin{pspicture}(5.32,5)(0,0.25)
\linethickness{1pt}

\multiput(3.16,0)(0.08,0){4}{\psline{-,linecolor=gray,linewidth=1pt}(0,0)(0,4.32)}
\multiput(3.64,0)(0.08,0){4}{\psline{-,linecolor=gray,linewidth=1pt}(0,0)(0,4.32)}
\multiput(4.6,0)(0.08,0){4}{\psline{-,linecolor=gray,linewidth=1pt}(0,0)(0,4.32)}
\multiput(5.08,0)(0.08,0){4}{\psline{-,linecolor=gray,linewidth=1pt}(0,0)(0,4.32)}

\multiput(1,2.16)(0,0.08){4}{\psline{-,linecolor=gray,linewidth=1pt}(0,0)(2.16,0)}
\multiput(1,2.64)(0,0.08){4}{\psline{-,linecolor=gray,linewidth=1pt}(0,0)(2.16,0)}
\multiput(1,3.6)(0,0.08){4}{\psline{-,linecolor=gray,linewidth=1pt}(0,0)(2.16,0)}
\multiput(1,4.08)(0,0.08){4}{\psline{-,linecolor=gray,linewidth=1pt}(0,0)(2.16,0)}

\multiput(2.08,0)(0.04,0){4}{\psline{-,linecolor=gray,linewidth=1pt}(0,0)(0,2.16)}
\multiput(2.32,0)(0.04,0){4}{\psline{-,linecolor=gray,linewidth=1pt}(0,0)(0,2.16)}
\multiput(2.8,0)(0.04,0){4}{\psline{-,linecolor=gray,linewidth=1pt}(0,0)(0,2.16)}
\multiput(3.04,0)(0.04,0){4}{\psline{-,linecolor=gray,linewidth=1pt}(0,0)(0,2.16)}

\multiput(1,1.08)(0,0.04){4}{\psline{-,linecolor=gray,linewidth=1pt}(0,0)(1.08,0)}
\multiput(1,1.32)(0,0.04){4}{\psline{-,linecolor=gray,linewidth=1pt}(0,0)(1.08,0)}
\multiput(1,1.8)(0,0.04){4}{\psline{-,linecolor=gray,linewidth=1pt}(0,0)(1.08,0)}
\multiput(1,2.04)(0,0.04){4}{\psline{-,linecolor=gray,linewidth=1pt}(0,0)(1.08,0)}

\rput(3.16,4.32){\psline{-,linecolor=gray,linewidth=1pt}(0,0)(2.175,0)}

\end{pspicture}
&
\begin{pspicture}(4,5)(0,0.25)
\linethickness{1pt}

\multiput(1,2.16)(0,0.08){4}{\psline{-,linecolor=gray,linewidth=1pt}(0,0)(2.16,0)}
\multiput(1,2.64)(0,0.08){4}{\psline{-,linecolor=gray,linewidth=1pt}(0,0)(2.16,0)}
\multiput(1,3.6)(0,0.08){4}{\psline{-,linecolor=gray,linewidth=1pt}(0,0)(2.16,0)}
\multiput(1,4.08)(0,0.08){4}{\psline{-,linecolor=gray,linewidth=1pt}(0,0)(2.16,0)}

\multiput(2.08,0)(0.04,0){4}{\psline{-,linecolor=gray,linewidth=1pt}(0,0)(0,2.16)}
\multiput(2.32,0)(0.04,0){4}{\psline{-,linecolor=gray,linewidth=1pt}(0,0)(0,2.16)}
\multiput(2.8,0)(0.04,0){4}{\psline{-,linecolor=gray,linewidth=1pt}(0,0)(0,2.16)}
\multiput(3.04,0)(0.04,0){4}{\psline{-,linecolor=gray,linewidth=1pt}(0,0)(0,2.16)}

\multiput(1,1.08)(0,0.04){4}{\psline{-,linecolor=gray,linewidth=1pt}(0,0)(1.08,0)}
\multiput(1,1.32)(0,0.04){4}{\psline{-,linecolor=gray,linewidth=1pt}(0,0)(1.08,0)}
\multiput(1,1.8)(0,0.04){4}{\psline{-,linecolor=gray,linewidth=1pt}(0,0)(1.08,0)}
\multiput(1,2.04)(0,0.04){4}{\psline{-,linecolor=gray,linewidth=1pt}(0,0)(1.08,0)}

\rput(3.16,0){\psline{-,linecolor=gray,linewidth=1pt}(0,0)(0,4.32)}

\end{pspicture}
&
\begin{pspicture}(4,5)(0,0.25)
\linethickness{1pt}
\rput(1.16,4.32){\psline{-,linecolor=blue,linewidth=2pt}(0,0)(2.16,0)}
\end{pspicture}

\end{tabular}
\caption{A depiction of $X_3, X_4$, the largest fibers, and the quotients $X_3/\!\sim$ and $X_4/\!\sim$ .}\label{scale-4}
\end{figure}
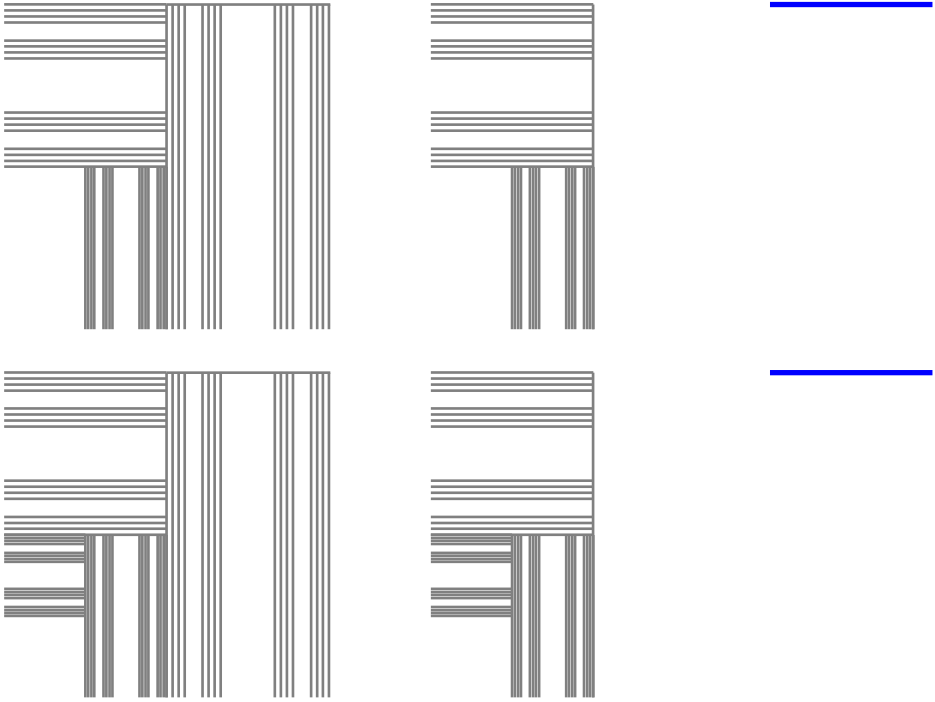
Let $X_4=X_2\cup\frac{X_2}{2}$. Then the largest fiber of $X_4$ is $F_{(1,2)}^*=X_4\cap\{(x_1,x_2): x_1\le 1\}$, which is homeomorphic with $X_3$. Similarly, we can infer that $\ell^*(X_4)=4$ and that $X_4/\!\sim$ is homeomorphic with $[0,1]$. See lower part of Figure \ref{scale-4}. The construction of $X_k$ for $k\ge5$ can be done inductively. The general formula
$X_{k+2}=X_2\bigcup\frac{1}{2}X_k$ defines a path-connected continuum for all $k\ge3$, for which the largest fiber is homeomorphic to $X_{k+1}$. Therefore, we have $\ell^*(X_k)=k$; moreover, the quotient space $X_k/\!\sim$ is always homeomorphic to the interval $[0,1]$. Finally, we can verify that \[X_\infty=\{(0,0)\}\cup\left(\bigcup_{k=2}^\infty X_k\right)\]
is a path connected continuum and that its largest fiber is homeomorphic to $X_\infty$ itself. Therefore, $X_\infty$ has a scale $\ell^*(X_\infty)=\infty$, and its quotient is homeomorphic to $[0,1]$.
\end{exam}

\textbf{Acknowledgements}

The authors are grateful to the referee for very helpful remarks, especially those about a gap in the proof for Theorem \ref{E=F} and an improved proof for Lemma \ref{fact-1}. The first author was supported by the Agence Nationale de la Recherche (ANR) and the Austrian Science Fund
(FWF) through the project \emph{Fractals and Numeration} ANR-FWF I1136 and the FWF Project 22 855.
The second author was supported by the Chinese National Natural Science Foundation Projects 10971233 and 11171123.

\normalsize
\baselineskip=17pt
\bibliographystyle{plain}
\bibliography{biblio}

\end{document}